\documentclass{amsart}

\usepackage[utf8]{inputenc}
\usepackage[english]{babel}
\usepackage{mathtools} 
\usepackage{xcolor,booktabs} 
\usepackage{ragged2e} 
\usepackage{amsmath,amssymb,mathrsfs,amsthm}
\usepackage{braket} 
\usepackage{hyperref}
\usepackage{comment}
\usepackage{csquotes}
\usepackage{tikz-cd} 
\usepackage{faktor} 
\usepackage{enumitem}
\usepackage[colorinlistoftodos,
	backgroundcolor=yellow!20!white,
	prependcaption,
	textsize=scriptsize
	]
	{todonotes} 
\usepackage{xargs}
\newcommandx{\fz}[2][1=]{\todo[inline,linecolor=blue,backgroundcolor=blue!25,bordercolor=blue,#1, author = FABRIZIO]{#2}} 
\newcommandx{\fm}[2][1=]{\todo[inline,linecolor=purple,backgroundcolor=purple!25,bordercolor=blue,#1, author = FELIX]{#2}} 
\newcommandx{\jdj}[2][1=]{\todo[inline,linecolor=red,backgroundcolor=red!25,bordercolor=red,#1, author = JEAN-DAVID]{#2}} 

\usepackage[backend=biber,style=alphabetic,sorting=ynt]{biblatex}
\theoremstyle{definition}
\newtheorem{definition}{Definition}[section]
\newtheorem{remark}[definition]{Remark}
\newtheorem{example}[definition]{Example}
\theoremstyle{plain}
\newtheorem{theorem}[definition]{Theorem}
\newtheorem*{theorem*}{Theorem}
\newtheorem{lemma}[definition]{Lemma}
\newtheorem{proposition}[definition]{Proposition}
\newtheorem*{proposition*}{Proposition}

\newcommand{\K}{\mathbb{K}}
\newcommand{\R}{\mathbb{R}}
\newcommand{\N}{\mathbb{N}}
\newcommand{\C}{\mathbb{C}}
\newcommand{\norm}[1]{\left\lVert #1\right\rVert} 
\newcommand{\id}{\operatorname{id}}
\newcommand{\Meas}{\mathbf{Meas}}
\newcommand{\opMeas}{\mathbf{Meas}^\mathrm{op}}
\newcommand{\embMeas}{\mathbf{Meas}_\mathrm{emb}}
\newcommand{\Vect}{\mathbf{Vect}_{\mathbb{K}}}

\newcommand{\NVect}{\mathbf{NVect}_{\mathbb{K}}}
\newcommand{\NVECT}{\mathbf{NVECT}_{\mathbb{K}}}
\newcommand{\NVembp}{\mathbb{NV}_{\mathrm{emb}}^p}
\newcommand{\NVp}{\mathbb{NV}^p}
\newcommand{\Ban}{\mathbf{Ban}_{\mathbb{K}}}
\newcommand{\BAN}{\mathbf{BAN}_{\mathbb{K}}}
\newcommand{\Bembp}{\mathbb{B}_{\mathrm{emb}}^p}
\newcommand{\Bp}{\mathbb{B}^p}

\newcommand{\proj}{\mathrm{proj}_{2}}
\newcommand{\Bemb}{\mathbb{B}_{\mathrm{emb}}^1}
\newcommand{\B}{\mathbb{B}^1}

\begin{document}

\title[Bochner integral]{A universal characterisation of the Bochner integral}

\author{Jean\--David Jacques}
\address{Institut f\"ur Mathematik, Universit\"at Potsdam}
\email{jean-david.jacques@uni-potsdam.de}
\author{Felix Medwed}
\address{Institut f\"ur Mathematik, Universit\"at Potsdam}
\email{felix.medwed@uni-potsdam.de}
\author{Fabrizio Zanello}
\address{Institut f\"ur Mathematik, Universit\"at Potsdam}
\email{fabrizio.zanello@uni-potsdam.de}

\begin{abstract}
    We provide a universal characterisation of the Banach spaces of equivalence classes of $L^p$ Bochner integrable functions, under equality almost everywhere, as initial objects in appropriate categories of bifunctors, for all $p \in [1,\infty)$.
    As a by-product of our construction, we also obtain a unique characterisation of the Bochner integral. 
\end{abstract}

\maketitle

\tableofcontents

\section{Introduction}

    The Bochner integral is a fundamental tool in functional analysis and operator theory, with a wide range of applications in harmonic analysis and probability. \\
    In our approach, we build on the work of Tom Leinster \cite{Leinster2023}, who provided a universal characterisation of the Banach spaces of $L^p$ Lebesgue integrable functions with values in $\K = \R,\C$ as initial objects of certain categories. \\
    The main challenge consists of generalising the construction developed there to the case of functions taking values in multi-dimensional vector spaces. 
    The naïve idea would be to replace $\K$ by some $\K$-vector space $V$ in Leinster's construction, but this strategy fails for at least three reasons.
    In the first place, because for generic vector spaces there is no canonical choice of a basis, in the second place, because the choosing a basis is not a functorially well-defined operation and, in the third place, because using indicator functions with values in the elements of a chosen basis as a genereting set yields a norm on the space of simple functions which depends on the choice of the basis. \\ 
    These issues require a more elaborate categorical framework and an accurate choice of additional structure to impose on the objects under consideration.

    In addition, a further motivation to consider the Bochner integral, comes from our interest in a universal characterisation of the signature of paths with values in Banach spaces (see \cite{Chen1957} for the original article and \cite{Chevyrev/Kormilitzin2016} for a recent overview) to which we plan to devote a future project.

    \subsection{The Bochner integral}

        The central idea at the base of Bochner's original work \cite{Bochner1933} is simple but fundamental: scalar integration is extended to Banach-space-valued functions by approximation with simple functions, with convergence understood in norm. 
        The resulting integration theory is the natural extension of the Lebesgue integral and retains most of its structural properties while providing a natural framework for integration in infinite-dimensional spaces.

        Consider a measure space $(X,\Sigma,\mu)$ and a Banach space $V$ over the field $\K = \R,\C$. 
        A function $f \colon X \rightarrow V$ is called \emph{strongly measurable} if it is almost everywhere the pointwise limit of $V$-valued simple functions (see \cite{Cohn2013} for more details). 
        The function $f$ is Bochner integrable if it is strongly measurable and satisfies
        \begin{equation*}
            \int_X \norm{f(x)}_V d\mu(x) < \infty.
        \end{equation*}
        The space of Bochner integrable functions defined on $X$, with values in $V$, is denoted by $L^1(X,V)$. \\
        In order to define the Bochner integral of an element $f \in L^1(X,V)$, one proceeds as follows.
        First, suppose that $g \colon X \rightarrow V$ is a simple and Bochner integrable function. 
        Let $v_1,\dots,v_n \in V$ be the non-zero values of $g$, and suppose that these values are attained on measurable subsets $A_1,\dots,A_n \in \Sigma$. 
        Then, integrability implies that each $A_k$ has finite measure with respect to $\mu$ and it is possible to define        
        \begin{equation*}
            \int_X g(x) d\mu(x) = \sum_{k=1}^n \mu(A_k) v_k.
        \end{equation*}
        As a direct consequence of the definition, we have that
        \begin{equation*}
            \norm{\int_X g(x) d\mu(x)}_V \leq \int_X \norm{g(x)}_V d\mu(x)
        \end{equation*}
        Now suppose $f$ is an arbitrary Bochner integrable function.
        We can choose a sequence $(g_n)_{n \in \N}$ of simple Bochner integrable functions such that $g_n(x) \xrightarrow[n \rightarrow \infty]{} f(x)$ for almost every $x \in X$ and such that the $\K$-valued function $x \mapsto \sup_{n \in \N}\norm{g_n(x)}_V$ is integrable.
        The dominated convergence theorem for $\K$-valued functions implies that $\lim_{n \rightarrow \infty}\int_X \norm{f(x) - g_n(x)}_V d\mu(x) = 0$.
        Thus $\left( \int_X g_n(x) d\mu(x) \right)_{n \in \N}$ is a Cauchy sequence in $V$ and hence it converges.
        The Bochner integral $\int_X f(x) d\mu(x)$ is defined as the limit of this sequence and it is easy to check that it does not depend on the choice of the approximating sequence $(g_n)_{n \in \N}$.

        One of the basic properties of the Bochner integral (see \cite{Hytoenen/van_Neerven/Veraar/Weis2016} for a systematic exposition) is that, for any $f \in L^1(X,V)$ and for any bounded linear operator $\varphi \colon V \rightarrow W$, we have the equality
        \begin{equation}
        \label{eq:linear forms}
            \int_X \left(\varphi \circ f \right)(x) d\mu(x) = \varphi \left( \int_X f(x) d\mu(x) \right).
        \end{equation}        
        In particular, the last equality remains valid in $\K$ for all continuous linear functionals $\varphi \in V^*$, which, using a well-known corollary of the Hahn-Banach theorem, implies that the Bochner integral of any $f \in L^1(X,V)$ is uniquely characterised by the values of the Lebesgue integrals $\int_X \left(\varphi \circ f \right)(x) d\mu(x)\in \K$ for $\varphi$ ranging in $V^*$.
        Conversely, given a general Banach space $V$ and a linear functional $F \in V^{**}$, the existence of an element $I \in V$, such that $F(\varphi) = \varphi(I)$ is not guaranteed, since in general the canonical injection $V\hookrightarrow V^{**}$ is not surjective. 
        This implies, in particular, that for any strongly measurable $f \colon X \rightarrow V$, the bounded linear functional $F \colon \varphi \mapsto \int_X \left(\varphi \circ f \right)(x) d\mu(x)$ does not \textit{a priori} guarantee the existence of an element $I_f\in V$ such that $F(\varphi)=\varphi(I_f)$, which would correspond to the Bochner integral of the funtion $f$.

        This last observation justifies the relevance of Bochner's construction and leads to an interesting classification of Banach spaces, based on whether they satisty the $\mu$-Radon-Nikodym property or not. 
        The latter states that, for any measure $\nu$ on $X$, such that $\nu \ll \mu$, there exists a Bochner integrable funtion $f \in L^1(X,\R)$ such that, for every measurable set $A \in \Sigma$, one has
        \begin{equation*}
            \nu(A) = \int_A f(x) d\mu(x).
        \end{equation*}
        Further developments within the theory of vector measures and, in particular, from the point of view of Radon-Nikodym properties for the Bochner integral and the geometry of Banach spaces can be found in \cite{Diestel/Uhl1977}.

        The applications of the Bochner integral are correspondingly broad. 
        In functional analysis, it provides the natural construction of the spaces $L^p(X,V)$ and is indispensable in the study of vector-valued operators and tensor-product representations (see \cite{Raymond2002} for a comprehensive treatment). 
        In differential equations, Banach-space-valued integration permits the formulation of integral and mild solutions of evolution equations and is closely connected with the theory of strongly continuous semigroups \cite{Hille/Phillips1957}. 
        In harmonic analysis, the Bochner integral provides a framework for extending Fourier-analytic methods to Banach and Hilbert spaces (see \cite{Hytoenen/van_Neerven/Veraar/Weis2016}).
        In probability theory, Banach-space-valued random variables have Bochner expectations, making the Bochner integral fundamental for the study of random elements in infinite-dimensional spaces (see, for example, \cite{Chatterji1968}). 
        This also permits to derive new notions of conditional expectations and consequently to study martingales with values in Banach spaces as done, for example, in the recent work \cite{steinwart2024conditioning}.

        We also mention, for the reader with practical views, the interesting fact that in classical examples, when the Banach space $(V,\norm{\cdot})$ is a function space, say $C([0,1],\R)$, if, for a given $t \in[0,1]$, the Dirac operator $\delta_t \colon V \rightarrow \R \, \colon f \mapsto f_t$ is continuous with respect to the norm $\norm{\cdot}$, then equality \eqref{eq:linear forms}, with $\varphi = \delta_t$ gives:
        \begin{equation*}
            \left( \int_X f(x) d\mu(x) \right)(t) = \int_X  f_t (x) d\mu(x).
        \end{equation*}
        Again, in the case of $C([0,1],\R)$, the continuity of $\delta_t$ for all $t \in [0,1]$ is ensured by the sufficient condition that $\norm{\cdot}_\infty \leq \norm{\cdot}$.
        Therefore, in probability theory in the study of continuous semi-martingales, the Bochner integral can also lead to different notions of mean curve, depending on the norm considered on $C([0,1],\R)$.
        
        We also mention an extension of the Bochner integral, which was later provided in the work by Pettis \cite{Pettis1938}, based on the notion of duality for Banach spaces. 
        In modern terminology, a function $f \colon X \rightarrow V$ is Pettis integrable when the $\K$-valued function $\varphi(f)$ is integrable for every $\varphi \in V^*$, and there is an element $I_f \in V$ such that
        \begin{equation*}
            \varphi(I_f) = \int_X \varphi\left( f(x) \right) d\mu(x).
        \end{equation*}
        The Pettis integral therefore replaces the norm approximation central to Bochner integration by a dual, weak formulation.

    \subsection{Strategy and outline of the paper}

        We take Leinster's remarkably clear and subtle paper \cite{Leinster2023} as our starting point and try to retain its structure as much as possible.
        There, the Lebesgue integral of $\K$-valued functions, where $\K = \R, \C$, defined on measure spaces with finite total measure is characterised in three steps:
        \begin{enumerate}
            \item universal characterisation of the vector spaces of simple functions;
            \item universal characterisation of the normed vector spaces of equivalence classes of simple functions, under equality almost everywhere;
            \item completion of the normed vector spaces of equivalence classes of simple functions to obtain the Banach spaces of equivalence classes of Lebesgue integrable functions, under equality almost everywhere.
        \end{enumerate}
        These steps correspond to the introduction of three increasingly specific categories of functors (with a little additional structure) from the category of finite measure spaces and embeddings $\embMeas$ (or, alternatively, the opposite of the category of finite measure spaces and measure-preserving partial maps $\opMeas$), respectively, to the categories $\Vect$, of vector spaces and linear maps, $\NVect$, of normed vector spaces and linear contractive maps, and $\Ban$, of Banach spaces and linear contractive maps. \\
        In this framework, the main result of the first step is that the functor associating, to each finite measure space, the vector space of $\K$-valued simple functions on it is the initial object of the respective category. 
        The main result of the second step crucially uses the fact that the category of the second step can be regarded as a subcategory of the category of the first step to show that the functor associating, to each finite measure space, the normed vector space of equivalence classes of $\K$-valued simple functions on it, under equality almost everywhere, is the initial object in the respective category. 
        Finally, the third step uses the fact that the forgetful functor from the category of the third step to the category of the second step admits a left adjoint given by the completion functor to compute the initial object of the category of the third step as the image, via the completion functor, of the initial object of the category of the second step.
        This precisely yields the functor of Lebesgue integrable functions as the completion of the functor of equivalence classes of simple functions. 
        
        Since our ultimate aim is to characterise the spaces of integrable functions taking values in any Banach space, and not only in $\K$, we are naturally led to introduce an additional dependence in the functors considered in \cite{Leinster2023}, accounting for the choice of codomain of the functions. \\
        The natural way to adapt Leinster's strategy to the case of vector-valued functions would then be to consider three categories of bifunctors (with a little additional structure) -- respectively, an appropriate subcategory of $\Vect^{\embMeas \times \Vect}$ for the analogous version of the first step, of $\NVect^{\embMeas \times \NVect}$ for the analogous version of the second step and of $\Ban^{\embMeas \times \Ban}$ for the analogous version of the third step -- and to try to generalise the corresponding results.
        However, this plan would immediately collapse due to the fact that the category of bifunctors of the second step cannot be regarded as a subcategory of the category of the first step. 
        In particular, there is no canonical way to extend a bifunctor defined on $\embMeas \times \NVect$ to a bifunctor defined on $\embMeas \times \Vect$, because there is no canonical way to assign a norm to any vector space. \\
        Fortunately, the fact that the categories $\NVect$ and $\Ban$ can be related by canonical adjoint functors implies that also the categories of bifunctors of the second and third steps can be related by canonical adjoint functors.
        This is, indeed, enough to be able to generalise Leinster's construction in two steps:
        \begin{enumerate}
            \item universal characterisation of the normed vector spaces of equivalence classes of simple functions with values in normed vector spaces, under equality almost everywhere, see Section \ref{sec:simple functions with values in normed vector spaces};
            \item completion of the normed vector spaces of equivalence classes of simple functions with values in normed vector spaces to the Banach spaces of equivalence classes of $L^p$ Bochner integrable functions, under equality almost everywhere, for any $p \in [1,\infty)$, see Section \ref{sec:completion}.
        \end{enumerate}

        More specifically, the outline of the paper is as follows.
        In Section \ref{sec:basic definitions examples and remarks 1}, we recall basic definitions and useful remarks from \cite{Leinster2023} and then introduce the categories of bifunctors (with a little additional structure) $\NVembp \subset \NVect^{\embMeas \times \NVect}$ (see Definition \ref{def:NV_embp}) and $\NVp \subset \NVect^{\opMeas \times \NVect}$ (see Definition \ref{def:NVp}).
        The additional structure with which we endow the categories $\NVembp$ and $\NVp$ is the crucial ingredient that allows us to adapt in a natural way the delicate constructions and arguments used in \cite{Leinster2023} to our setting. 
        Roughly speaking, it consists, on the one hand, of the fact that objects in these categories are pairs $(F,c)$ where $F$ is a bifunctor and $c$ denotes the assignment, to every object of the domain category of the bifunctor $F$, of a bounded linear map.
        On the other hand, the additional structure that we impose is represented by the requirement that some axioms, involving both the bifunctors and the assignments, must be satisfied. \\ 
        Besides introducing the categories $\NVembp$ and $\NVp$, for all $p \in [1,\infty]$, we also discuss in detail the features of the two fundamental examples (denoted, by abuse of notation, with the same symbol) of objects $(S^p,\sigma) \in \NVembp$ and $(S^p,\sigma) \in \NVp$ (see, respectively, Example \ref{ex:simple functions with values in normed vector spaces 1} and Example \ref{ex:simple functions with values in normed vector spaces 2}).
        Here, $S^p$ denotes the bifunctor associating, to any measure space $X$ and to any normed vector space $V$, the normed vector space $S^p(X,V)$ of equivalence classes of simple functions on $X$ with values in $V$, under equality almost everywher, endowed with the $L^p$ norm.
        In addition, $\sigma$ assigns, to any measure space $X$ and to any normed vector space $V$, a bounded linear map $\sigma_{(X,V)} \colon V \rightarrow S^p(X,V)$ defined, for any element $v \in V$, as the constant function on $X$ with value $v$.

        In Section \ref{sec:preparatory lemmas and main result}, we first prove four technical lemmas -- Lemma \ref{lemma:Beck-Chevalley conditions in the normed case}, Lemma \ref{lemma:functors and pairwise disjoint embeddings in the normed case}, Lemma \ref{lemma:functors and measure-preserving maps in the normed case} and Lemma \ref{lemma:compatibility of c with embeddings in the normed case} -- which we then use to obtain the proof of the first main result of this work.
        \begin{theorem*}[\ref{thm:universal properties of simple functions with values in normed vector spaces}]
            Let $p \in [1,\infty]$. 
            The bifunctors of the normed vector spaces of equivalence classes of simple functions with values in normed vector spaces, under equality almost everywhere, have the following universal properties:
            \begin{itemize}
                \item[(i)] $(S^p,\sigma)$ is the initial object of the category $\NVembp$;
                \item[(ii)] $(S^p,\sigma)$ is the initial object of the category $\NVp$.
            \end{itemize}
        \end{theorem*}
        In Section \ref{sec:basic definitions examples and remarks 2}, we start by introducing the appropriate categories of bifunctors (with a little additional structure) $\Bembp \subset \Ban^{\embMeas \times \Ban}$ (see Definition \ref{def:Bemb}) and $\Bp \subset \Ban^{\opMeas \times \Ban}$ (see Definition \ref{def:Bp}), for all $p \in [1,\infty]$. 
        The additional structure with which we endow the categories $\Bembp$ and $\Bp$ is essentially the same as before. \\
        We also discuss in detail the features of the two fundamental examples (denoted again, by abuse of notation, with the same symbol) of objects $(L^p,\sigma) \in \Bembp$ and $(L^p,\sigma) \in \Bp$ (see, respectively, Example \ref{ex:Lpemb functor} and Example \ref{ex:Lp functor}).
        Here, roughly speaking, $L^p$ denotes the bifunctor associating, to any measure space $X$ and to any Banach space $V$, the Banach space $L^p(X,V)$ of equivalence classes of $L^p$ Bochner integrable functions on $X$ with values in $V$, under equality almost everywhere.
        In addition, $\sigma$ is defined in the same way as above. \\
        After this, and in preparation for the subsequent results, in Remark \ref{rem:on all the forgetful functors}, we point out the existence of some useful functors between the categories $\NVembp$, $\NVp$, $\Bembp$ and $\Bp$ and describe in detail their features.

        These observations constitute the starting point, in Section \ref{sec:main results}, for the proof of the second main result of this work.
        \begin{theorem*}[\ref{thm:universal properties of Bochner integrable functions}]
            Let $p \in [1,\infty)$. The bifunctors of the Banach spaces of equivalence classes of $L^p$ Bochner integrable functions, under equality almost everywhere, have the following universal properties:
            \begin{itemize}
                \item[(i)] $(L^p,\sigma)$ is the initial object of the category $\Bembp$;
                \item[(ii)] $(L^p,\sigma)$ is the initial object of the category $\Bp$.
            \end{itemize} 
        \end{theorem*} 
        In Remark \ref{rem:issue with the infinity case}, we motivate and show with a concrete example why the case $p = \infty$ cannot be covered by Theorem \ref{thm:universal properties of Bochner integrable functions} and finally, as a last step, we show how the universal characterisation of the Banach spaces of equivalence classes of $L^p$ Bochner integrable functions, under equality almost everywhere, naturally induces a unique characterisation of the Bochner integral. \\
        To this end, we fix $p = 1$ and introduce a distinguished object $(\proj,\tau)$ in the category $\Bemb$.
        We discuss its features in detail and then we give the proof of the third main result of this work.
        \begin{proposition*}[\ref{prop:uniqueness of the Bochner integral}]
            The unique morphism $(L^1,\sigma) \rightarrow (\proj,\tau)$ in the category $\Bemb$ is given, component-wise, by the Bochner integral:
            \begin{equation*}
                \int_{(X,V)} = \int_X \colon L^1(X,V) \rightarrow V,
            \end{equation*}
            for any object $(X,V) \in \embMeas \times \Ban$.
        \end{proposition*}

    \subsection{Notation and conventions}
    
        For the reader's convenience, we fix here some notation and conventions regarding categories, functors and natural transformations which we shall use throughout the paper. \\
        Consider two categories $\mathbf{A}$ and $\mathbf{B}$ and a functor $F \colon \mathbf{A} \rightarrow \mathbf{B}$.
        The action of the functor $F$ on any object $X \in \mathbf{A}$ is denoted by $FX \in \mathbf{B}$.
        The action of the functor $F$ on any morphism $a \colon X \rightarrow Y$ in the category $\mathbf{A}$ yields a morphism denoted by $F_a \colon FX \rightarrow FY$ in the category $\mathbf{B}$. \\
        Given another functor $G \colon \mathbf{A} \rightarrow \mathbf{B}$, a natural transformation from $F$ to $G$ is denoted by $\varphi \colon F \Rightarrow G$ and its components determine morphisms $\varphi_X \colon FX \rightarrow GX$, for any object $X \in \mathbf{A}$, in the category $\mathbf{B}$. \\
        The opposite category of the category $\mathbf{A}$ is denoted by $\mathbf{A}^{\mathrm{op}}$.
        Objects in $\mathbf{A}^{\mathrm{op}}$ are denoted in the same way as objects in $\mathbf{A}$, while a morphism in $\mathbf{A}^{\mathrm{op}}$ is denoted by $m^{\mathrm{op}} \colon X \rightarrow Y$, where $m \colon Y \rightarrow X$ is the corresponding morphism in the category $\mathbf{A}$. \\
        Finally, we will also need to consider product categories.
        Objects in the product category $\mathbf{A} \times \mathbf{B}$ are pairs $(X,Z)$, where $X \in \mathbf{A}$ and $Z \in \mathbf{B}$.
        Morphisms are pairs $(a,b) \colon (X,Z) \rightarrow (X',Z')$, where $a \colon X \rightarrow X'$ is a morphism in $\mathbf{A}$, while $b \colon Z \rightarrow Z'$ is a morphism in $\mathbf{B}$.
        The composition of morphisms $(a,b) \colon (X,Z) \rightarrow (X',Z')$ and $(c,d) \colon (X',Z') \rightarrow (X'',Z'')$ in $\mathbf{A} \times \mathbf{B}$ is always intended component-wise, namely, $(c,d) \circ (a,b) = (c \circ a,d \circ b) \colon (X,Z) \rightarrow (X'',Z'')$.

\section{Simple functions with values in normed vector spaces}
\label{sec:simple functions with values in normed vector spaces}

    In this section, we provide a universal characterisation of the normed vector spaces of equivalence classes of simple functions with values in normed vector spaces, under equality almost everywhere.

    \subsection{Basic definitions, examples and remarks}
    \label{sec:basic definitions examples and remarks 1}

        We first recall some elementary notions from \cite{Leinster2023}.
        Then, we introduce the two categories $\NVembp$ and $\NVp$ in which the universal characterisation of the normed vector spaces of equivalence classes of simple functions with values in normed vector spaces, under equality almost everywhere, takes place.
        We discuss in full detail the fundamental examples of objects $(S^p,\sigma) \in \NVembp$ and $(S^p,\sigma) \in \NVp$, and we also include a number of useful remarks.
        Whenever possible, we adopt the same notation as in \cite{Leinster2023}, to improve the readability and facilitate the comparison.

        \begin{definition}
        \label{def:embedding}
            An \emph{embedding} of measure spaces $\iota \colon X \rightarrow Y$ is an injective map such that $A \subseteq X$ is measurable if and only if $\iota(A) \subseteq Y$ is measurable and such that $\mu_X(A) = \mu_Y \left( \iota(A) \right)$.
            We denote by $\embMeas$ the category whose objects are measure spaces with finite total measure and whose morphisms are embeddings.
        \end{definition}

        \begin{definition}
        \label{def:complementary embeddings}
            Two embeddings of measure spaces
            \begin{equation*}
                Y \xrightarrow{~~\alpha~~} X \xleftarrow{~~\beta~~} Z
            \end{equation*}
            are called \emph{complementary} if $\alpha(Y) \cap \beta(Z) = \emptyset$ and $\alpha(Y) \cup \beta(Z) = X$.
        \end{definition}

        \begin{definition}
            A \emph{measure-preserving partial map} $(D,p) \colon X \rightarrow Y$ is a measurable subset $D \subseteq X$ together with a measure-preserving map $p \colon D \rightarrow Y$, where $D$ is given the unique measure space structure such that the inclusion $D \hookrightarrow X$ is an embedding.
            Given two measure-preserving partial maps $(D,p) \colon X \rightarrow Y$ and $(E,q) \colon Y \rightarrow Z$, their composition is defined as $\left( p^{-1}(E),q \circ p \right) \colon X \rightarrow Z$.
            We denote by $\Meas$ the category whose objects are measure spaces with finite total measure and whose morphisms are measure-preserving partial maps. 
        \end{definition}

        \begin{remark}
        \label{rem:structure of morphisms in Meas}
            As pointed out in \cite[\S3]{Leinster2023}, any embedding $\iota \colon X \rightarrow Y$ determines a morphism in the category $\Meas$ denoted by $\left( \iota(X),\iota^{-1} \right) \colon Y \rightarrow X$. 
            In particular, this implies that to every measurable subset $A \subseteq X$ we can associate a morphism $(A,\id_A) \colon X \rightarrow A$ in $\Meas$.
            Furthermore, every measure-preserving map $m \colon X \rightarrow Y$, being a special instance of a measure-preserving partial map, determines a morphism $(X,m) \colon X \rightarrow Y$ in $\Meas$. \\
            These elementary facts correspond to the existence of two faithful and injective-on-objects functors (i.e.~embeddings in the categorical sense) as follows: 
            \begin{equation*}
                \begin{tikzcd}
                    \embMeas \ar[dr, hook]      & \\
                            & \opMeas \\
                    \Meas_{\mathrm{pres}}^{\mathrm{op}} \ar[ur, hook]      & 
                \end{tikzcd}
            \end{equation*}
            where $\Meas_{\mathrm{pres}}$ denotes the category whose objects are measure spaces with finite total measure and whose morphisms are measure-preserving maps. \\
            Moreover, any measure-preserving partial map $(D,p) \colon X \rightarrow Y$ can be canonically expressed as the composition of the morphism $(D,\id_D) \colon X \rightarrow D$, associated to the embedding $D \hookrightarrow X$, with the measure-preserving map $p \colon D \rightarrow Y$. \\
            The peculiar structure of the morphisms of the category $\opMeas$ has two relevant consequences concerning functors and natural transformations:
            \begin{itemize}
                \item[(i)] any functor from $\opMeas$ to any category $\mathbf{C}$ is completely determined by its action on objects, embeddings and measure-preserving maps;
                \item[(ii)] any transformation between two functors from $\opMeas$ to any category $\mathbf{C}$ is natural if and only if it is natural with respect to both embeddings and measure-preserving maps. 
            \end{itemize}
        \end{remark}

        \begin{definition}
            We denote by $\NVect$ the category whose objects are normed vector spaces over the field $\K$ (either $\R$ or $\C$) and whose morphisms are linear contractive maps.
            We also introduce the category $\NVECT$, whose objects are normed vector spaces over the field $\K$ (either $\R$ or $\C$) and whose morphisms are bounded linear maps.
        \end{definition}

        \begin{remark}
            Clearly, $\NVect$ is a subcategory of $\NVECT$.   
        \end{remark}

        We now have all the necessary ingredients to introduce the categories in which the universal characterisation of the normed vector spaces of equivalence classes of simple functions with values in normed vector spaces, under equality almost everywhere, will take place.

        \begin{definition}
        \label{def:NV_embp}
            Let $p \in [1,\infty]$. 
            We denote by $\NVembp$ the category whose objects are pairs $(F,c)$, where $F \colon \embMeas \times \NVect \rightarrow \NVect$ is a bifunctor and $c$ assigns a bounded linear (i.e.~continuous) map $c_{(X,V)} \colon V \rightarrow F(X,V)$ to every object $(X,V) \in \embMeas \times \NVect$. 
            In addition, we require the following axioms to hold:
            \begin{equation}
            \tag{\textbf{O}}
            \label{ax:compatibility with linear contractive maps}
                F_{(\id_X,T)} \circ c_{(X,V)} = c_{(X,W)} \circ T,
            \end{equation}            
            for any object $X \in \embMeas$ and any morphism $T \colon V \rightarrow W$ in $\NVect$;
            \begin{equation}
            \tag{\textbf{I}}
            \label{ax:compatibility with complementary embeddings in the normed case}
                c_{(Y,V)}^{(X,V)} + c_{(Z,V)}^{(X,V)} = c_{(X,V)},
            \end{equation}
            for any morphisms $(\alpha,\id_V) \colon (Y,V) \rightarrow (X,V)$ and $(\beta,\id_V) \colon (Z,V) \rightarrow (X,V)$ in $\embMeas \times \NVect$, where $Y \xrightarrow{~\alpha~} X \xleftarrow{~\beta~} Z$ are a couple of complementary embeddings and the notation is explained in Remark \ref{rem:notation for the maps c} below;
            \begin{equation} 
            \tag{\textbf{III}}
            \label{ax:boundedness of c}
                \norm{c_{(X,V)}}_{\mathcal{L}(V,F(X,V))} \leq \mu_X(X)^{\frac{1}{p}}, \qquad \forall (X,V) \in \embMeas \times \NVect, 
            \end{equation}
            where we recall that the operator norm of $c_{(X,V)}$ is defined as
            \begin{equation*}
                \norm{c_{(X,V)}}_{\mathcal{L}(V,F(X,V))} = \inf \Set{\epsilon \geq 0 | \norm{c_{(X,V)}(v)}_{F(X,V)} \leq \epsilon \norm{v}_V, \quad \forall v \in V},
            \end{equation*}
            and, in the case $p = \infty$, the expression $\mu_X(X)^{\frac{1}{p}}$ is intended as $0$ if $\mu_X(X) = 0$ and as $1$ otherwise;
            \begin{equation}
            \tag{\textbf{IV}}
            \label{ax:norm inequality wrt complementary embeddings}
                \norm{F_{(\alpha,\id_V)}u + F_{(\beta,\id_V)}w}_{F(X,V)} \leq \left( \norm{u}_{F(Y,V)}^p + \norm{w}_{F(Z,V)}^p \right)^{\frac{1}{p}},
            \end{equation}
            for any morphisms $(\alpha,\id_V) \colon (Y,V) \rightarrow (X,V)$ and $(\beta,\id_V) \colon (Z,V) \rightarrow (X,V)$ in $\embMeas \times \NVect$, where $Y \xrightarrow{~\alpha~} X \xleftarrow{~\beta~} Z$ are a couple of complementary embeddings, for all $u \in F(Y,V)$ and $w \in F(Z,V)$ and where, in the case $p = \infty$, the right-hand side of the inequality is intended as $\max \Set{\norm{u}_{F(Y,V)},\norm{w}_{F(Z,V)}}$. \\
            A morphism $\psi \colon (F,c) \rightarrow (G,d)$ in the category $\NVembp$ is a natural transformation $\psi \colon F \Rightarrow G$ such that $\psi_{(X,V)} \circ c_{(X,V)} = d_{(X,V)}$, for all $(X,V) \in \embMeas \times \NVect$.
        \end{definition}

        \begin{remark}
        \label{rem:notation for the maps c}
            In axiom \ref{ax:compatibility with complementary embeddings in the normed case} and throughout the paper, for any object $(F,c) \in \NVembp$, we adopt a specific notation to denote the action of embeddings of measure spaces on the assignment $c$, via the bifunctor $F$.
            More precisely, given any morphism $(\iota,\id_V) \colon (X,V) \rightarrow (Y,V)$ in the category $\embMeas \times \NVect$, where $\iota \colon X \rightarrow Y$ is an embedding, we denote
            \begin{equation*}
                c_{(X,V)}^{(Y,V)} := F_{(\iota,\id_V)} \circ c_{(X,V)}.
            \end{equation*}
            There is a slight abuse in this notation, in the fact that the map embedding $X$ into $Y$ is not immediately evident.
            Nevertheless, which embedding of measure spaces is being considered should always be clear from the context. 
        \end{remark}

        \begin{example}
        \label{ex:simple functions with values in normed vector spaces 1}
            The fundamental example of an object in the category $\NVembp$ is the pair $(S^p,\sigma)$. 
            Here, $S^p \colon \embMeas \times \NVect \rightarrow \NVect$ is the bifunctor which associates, to any object $(X,V) \in \embMeas \times \NVect$, the normed vector space $S^p(X,V)$ of equivalence classes of simple functions on $X$ with values in $V$, under equality almost everywhere. 
            The norm on $S^p(X,V)$ is defined, for $p \in [1,\infty)$ and for any $f \in S^p(X,V)$, as
            \begin{equation}
            \label{eq:Lp norm for simple functions with values in normed vector spaces}
                \norm{f}_{S^p(X,V)} = \left( \sum_{v \in V} \norm{v}_V^p \mu_{X}\left( f^{-1}(v) \right) \right)^{\frac{1}{p}},
            \end{equation}
            where, by definition of (equivalence class of) simple functions, only finitely many summands are non-zero and, for $p = \infty$, as
            \begin{equation}
            \label{eq:Linfty norm for simple functions with values in normed vector spaces}
                \norm{f}_{S^\infty(X,V)} = \inf \Set{\epsilon \geq 0 | \norm{f(x)}_V \leq \epsilon \quad \text{a.e.~on} ~ X}.
            \end{equation}
            The action of the bifunctor $S^p$ on a morphism $(\iota,T) \colon (X,V) \rightarrow (Y,W)$ in the category $\embMeas \times \NVect$ is defined informally by extension to zero outside the image of the embedding $\iota \colon X \rightarrow Y$ and by post-composition with the linear contractive map $T \colon V \rightarrow W$. More precisely, $S^p_{(\iota,T)} \colon S^p(X,V) \rightarrow S^p(Y,W)$ is defined, for any $f \in S^p(X,V)$, by choosing a representative (denoted, by abuse of notation, by the same symbol $f$) and setting, for almost every $y \in Y$:
            \begin{equation}
            \label{eq:action of Sp on emb-morphisms}
                \begin{split}
                    S^p_{(\iota,T)}f \colon y \mapsto 
                        \begin{cases}
                            (T \circ f \circ \iota^{-1})(y) & \text{if} ~ y \in \iota(X), \\
                            0 & \text{otherwise,}
                        \end{cases}
                \end{split}
            \end{equation}
            where the fact that embeddings map measure-zero subsets into measure-zero subsets (see Definition \ref{def:embedding}) guarantees that this definition does not depend on the choice of the representative for $f$.
            That $S^p_{(\iota,T)}$ is a contractive map can be checked, for $p \in [1,\infty)$ and for any $f \in S^p(X,V)$, as follows:
            \begin{align}
                \norm{S^p_{(\iota,T)} f}_{S^p(Y,W)} &= \left( \sum_{v \in V} \norm{T(v)}_W^p \mu_Y\left( \iota\left( f^{-1}(v) \right) \right) \right)^{\frac{1}{p}} \notag \\
                &= \left( \sum_{v \in V} \norm{T(v)}_W^p \mu_X\left( f^{-1}(v) \right) \right)^{\frac{1}{p}} \label{eq:using that iota is an embedding} \\
                & \leq \left( \sum_{v \in V} \norm{v}_V^p \mu_X\left( f^{-1}(v) \right) \right)^{\frac{1}{p}} \label{eq:using that T is a contractive map} \\
                &= \norm{f}_{S^p(X,V)}, \notag
            \end{align}
            where equation \eqref{eq:using that iota is an embedding} follows from the fact that $\iota \colon X \rightarrow Y$ is an embedding, while inequality \eqref{eq:using that T is a contractive map} follows from the fact that $T \colon V \rightarrow W$ is a contractive map.
            In the case $p = \infty$, we have:
            \begin{align}
                \norm{S^\infty_{(\iota,T)} f}_{S^\infty(Y,W)} &= \inf \Set{\epsilon \geq 0 | \norm{(T \circ f \circ \iota^{-1})(y)}_W \leq \epsilon \quad \text{a.e.~on} ~ \iota(X)} \notag \\
                &= \inf \Set{\epsilon \geq 0 | \norm{(T \circ f)(x)}_W \leq \epsilon \quad \text{a.e.~on} ~ X} \notag \\
                &\leq \inf \Set{\epsilon \geq 0 | \norm{f(x)}_V \leq \epsilon \quad \text{a.e.~on} ~ X} \label{eq:using again that T is a contraction} \\
                &= \norm{f}_{S^\infty(X,V)}, \notag
            \end{align}
            where inequality \eqref{eq:using again that T is a contraction} follows from the fact that $T \colon V \rightarrow W$ is a contractive map. \\
            For any object $(X,V) \in \embMeas \times \NVect$, the map $\sigma_{(X,V)} \colon V \rightarrow S^p(X,V)$ is defined by:
            \begin{equation}
            \label{eq:definition of sigma}
                \sigma_{(X,V)}(v) = v \cdot \chi_X,
            \end{equation}
            where $\chi_X$ denotes the (equivalence class, under equality almost everywhere, of the) constant function on $X$ with value $1$ (i.e.~the indicator function supported on $X$) and $\cdot$ denotes the natural extension of the scalar multiplication in $V$ to the ring of (equivalence classes, under equality almost everywhere, of) $\K$-valued simple functions on $X$. 
            Linearity of $\sigma_{(X,V)}$ is clear from the definition, while the fact that it is a bounded map is guaranteed, for the case $p \in [1,\infty)$, by the following basic equality: 
            \begin{equation*}
                \norm{\sigma_{(X,V)}(v)}_{S^p(X,V)} = \norm{v}_V \mu_X(X)^{\frac{1}{p}}, \qquad \forall v \in V.
            \end{equation*}
            Analogously, for the case $p = \infty$, boundedness of $\sigma_{(X,V)}$ follows directly from:
            \begin{equation*}
                \begin{cases}
                    \norm{\sigma_{(X,V)}(v)}_{S^\infty(X,V)} = \norm{v}_V   & \text{if} ~ \mu_X(X) \neq 0, \\
                    \norm{\sigma_{(X,V)}(v)}_{S^\infty(X,V)} = 0   & \text{if} ~ \mu_X(X) = 0,
                \end{cases}
            \end{equation*} 
            for all $v \in V$. The last two equations imply, respectively, that 
            \begin{equation*}
                \norm{\sigma_{(X,V)}}_{\mathcal{L}\left( V,S^p(X,V) \right)} = \mu_X(X)^{\frac{1}{p}}
            \end{equation*}
            and
            \begin{equation*}
                \begin{cases}
                    \norm{\sigma_{(X,V)}}_{\mathcal{L}\left( V,S^\infty(X,V) \right)} = 1   & \text{if} ~ \mu_X(X) \neq 0, \\
                    \norm{\sigma_{(X,V)}}_{\mathcal{L}\left( V,S^\infty(X,V) \right)} = 0   & \text{if} ~ \mu_X(X) = 0,
                \end{cases}
            \end{equation*}
            which, incidentally, correspond to axiom \ref{ax:boundedness of c} being satisfied. \\
            In order to check the validity of axiom \ref{ax:compatibility with linear contractive maps}, it suffices to apply formula \eqref{eq:action of Sp on emb-morphisms} to a morphism of the type $(\id_X,T) \colon (X,V) \rightarrow (X,W)$ in $\embMeas \times \NVect$. As expected, for all $v \in V$, we obtain:
            \begin{equation*}
                \begin{split}
                    \left( S^p_{(\id_X,T)} \circ \sigma_{(X,V)} \right)(v) &= S^p_{(\id_X,T)}(v \cdot \chi_X) \\
                    &= T(v) \cdot \chi_X \\
                    &= \left( \sigma_{(X,W)} \circ T \right)(v). 
                \end{split}
            \end{equation*}
            Axiom \ref{ax:compatibility with complementary embeddings in the normed case} follows from the elementary fact that the sum of indicator functions supported on disjoint subsets (up to measure-zero sets) gives the indicator function supported on the union of the respective supports (up to measure-zero sets). \\
            Finally, we have to check the validity of axiom \ref{ax:norm inequality wrt complementary embeddings}.
            For $p \in [1,\infty)$ and for any simple functions $f \in S^p{(Y,V)}$ and $g \in S^p{(Z,V)}$, denoting 
            \begin{equation*}
                h = S^p_{(\alpha,\id_V)}f + S^p_{(\beta,\id_V)}g \in S^p{(X,V)},
            \end{equation*}
            with $Y \xrightarrow{~\alpha~} X \xleftarrow{~\beta~} Z$ any couple of complementary embeddings, we obtain indeed:
            \begin{align}
                \norm{h}_{S^p(X,V)} &= \left( \sum_{v \in V} \norm{v}_V^p \mu_{X}\left( h^{-1}(v) \right) \right)^{\frac{1}{p}} \notag \\
                &= \left( \sum_{v \in V} \norm{v}_V^p \left( \mu_{X}\left( h^{-1}(v) \cap \alpha(Y) \right) + \mu_{X}\left( h^{-1}(v) \cap \beta(Z) \right) \right) \right)^{\frac{1}{p}} \label{eq:complementary embeddings and additivity of mu} \\
                &= \left( \sum_{v \in V} \norm{v}_V^p \mu_{Y}\left( f^{-1}(v) \right) + \sum_{v \in V} \norm{v}_V^p \mu_{Z}\left( g^{-1}(v) \right) \right)^{\frac{1}{p}} \label{eq:h coincides with f on Y and with g on Z} \\
                &= \left( \norm{f}_{S^p(Y,V)}^p + \norm{g}_{S^p(Z,V)}^p \right)^{\frac{1}{p}}, \notag
            \end{align}
            where equation \eqref{eq:complementary embeddings and additivity of mu} follows from the fact that $X = \alpha(Y) \cup \beta(Z)$ and $\alpha(Y) \cap \beta(Z) = \emptyset$ (see Definition \ref{def:complementary embeddings}) and from additivity of the measure $\mu_X$, while equation \eqref{eq:h coincides with f on Y and with g on Z} follows from the fact that $h$ coincides, respectively, with $f$ on $\alpha(Y)$ and with $g$ on $\beta(Z)$ and from the fact that embeddings preserve measures. 
            For $p = \infty$ and for any simple functions $f \in S^\infty{(Y,V)}$ and $g \in S^\infty{(Z,V)}$, denoting 
            \begin{equation*}
                h = S^\infty_{(\alpha,\id_V)}f + S^\infty_{(\beta,\id_V)}g \in S^\infty{(X,V)},
            \end{equation*}
            with $Y \xrightarrow{~\alpha~} X \xleftarrow{~\beta~} Z$ any couple of complementary embeddings, we obtain:
            \begin{align}
                \norm{h}_{S^\infty(X,V)} &= \inf \Set{\epsilon \geq 0 | \norm{h(x)}_V \leq \epsilon \quad \text{a.e.~on} ~ X} \notag \\
                &= \max \big\{ \inf \Set{\epsilon_1 \geq 0 ~|~ \norm{h(x)}_V \leq \epsilon_1 \quad \text{a.e.~on} ~ \alpha(Y)}, \notag \\
                &\qquad \qquad \inf \Set{\epsilon_2 \geq 0 ~|~ \norm{h(x)}_V \leq \epsilon_2 \quad \text{a.e.~on} ~ \beta(Z)} \big\} \notag \\
                &= \max \big\{ \inf \Set{\epsilon_1 \geq 0 ~|~ \norm{f(y)}_V \leq \epsilon_1 \quad \text{a.e.~on} ~ Y}, \label{eq:h coincides with f on Y and with g on Z and embeddings dont matter} \\
                &\qquad \qquad \inf \Set{\epsilon_2 \geq 0 ~|~ \norm{g(z)}_V \leq \epsilon_2 \quad \text{a.e.~on} ~ Z} \big\} \notag \\
                &= \max \Set{\norm{f}_{S^\infty(Y,V)},\norm{g}_{S^\infty(Z,V)}}, \notag
            \end{align}
            where equation \eqref{eq:h coincides with f on Y and with g on Z and embeddings dont matter} follows from the fact that $h$ coincides, respectively, with $f$ on $\alpha(Y)$ and with $g$ on $\beta(Z)$.
        \end{example}

        \begin{definition}
        \label{def:NVp}
            Let $p \in [1,\infty]$. 
            We denote by $\NVp$ the category whose objects are pairs $(F,c)$, where $F \colon \opMeas \times \NVect \rightarrow \NVect$ is a bifunctor and $c$ assigns a bounded linear (i.e.~continuous) map $c_{(X,V)} \colon V \rightarrow F(X,V)$ to every object $(X,V) \in \opMeas \times \NVect$. 
            In addition, we require the following axioms to hold:
            \begin{equation}
            \tag{\textbf{O*}}
            \label{ax:compatibility with linear contractive maps 1}
                F_{(\id_X,T)} \circ c_{(X,V)} = c_{(X,W)} \circ T,
            \end{equation}            
            for any object $X \in \opMeas$ and any morphism $T \colon V \rightarrow W$ in $\NVect$;
            \begin{equation}
            \tag{\textbf{I*}}
            \label{ax:compatibility with complementary embeddings in the normed case 1}
                c_{(Y,V)}^{(X,V)} + c_{(Z,V)}^{(X,V)} = c_{(X,V)},
            \end{equation}
            for any morphisms 
            \begin{equation*}
                \begin{split}
                    \left( \left( \alpha(Y),\alpha^{-1} \right)^{\mathrm{op}},\id_V \right) &\colon (Y,V) \rightarrow (X,V) \\
                    \left( \left( \beta(Z),\beta^{-1} \right)^{\mathrm{op}},\id_V \right) &\colon (Z,V) \rightarrow (X,V)
                \end{split}
            \end{equation*}
            in $\opMeas \times \NVect$, where $Y \xrightarrow{~\alpha~} X \xleftarrow{~\beta~} Z$ are a couple of complementary embeddings and the notation is explained in Remark \ref{rem:notation for the maps c 1} below; 
            \begin{equation}
            \tag{\textbf{II*}} 
            \label{ax:compatibility with measure-preserving maps in the normed case}
                F_{(m^{\mathrm{op}},\id_V)} \circ c_{(X,V)} = c_{(Y,V)},
            \end{equation}
            for any morphism $(m^{\mathrm{op}},\id_V) \colon (X,V) \rightarrow (Y,V)$ in the category $\opMeas \times \NVect$, where $m \colon Y \rightarrow X$ is a measure-preserving map;
            \begin{equation} 
            \tag{\textbf{III*}}
            \label{ax:boundedness of c 1}
                \norm{c_{(X,V)}}_{\mathcal{L}(V,F(X,V))} \leq \mu_X(X)^{\frac{1}{p}}, \qquad \forall (X,V) \in \opMeas \times \NVect, 
            \end{equation}
            where we recall that the operator norm of $c_{(X,V)}$ is defined as
            \begin{equation*}
                \norm{c_{(X,V)}}_{\mathcal{L}(V,F(X,V))} = \inf \Set{\epsilon \geq 0 | \norm{c_{(X,V)}(v)}_{F(X,V)} \leq \epsilon \norm{v}_V, \quad \forall v \in V},
            \end{equation*}
            and, in the case $p = \infty$, the expression $\mu_X(X)^{\frac{1}{p}}$ is intended as $0$ if $\mu_X(X) = 0$ and as $1$ otherwise;
            \begin{equation}
            \tag{\textbf{IV*}}
            \label{ax:norm inequality wrt complementary embeddings 1}
                \begin{split}
                    &\norm{F_{\left( \left( \alpha(Y),\alpha^{-1} \right)^{\mathrm{op}},\id_V \right)}u + F_{\left( \left( \beta(Z),\beta^{-1} \right)^{\mathrm{op}},\id_V \right)}w}_{F(X,V)} \leq \left( \norm{u}_{F(Y,V)}^p + \norm{w}_{F(Z,V)}^p \right)^{\frac{1}{p}},
                \end{split}
            \end{equation}
            for any couple $Y \xrightarrow{~\alpha~} X \xleftarrow{~\beta~} Z$ of complementary embeddings, for all $u \in F(Y,V)$ and $w \in F(Z,V)$ and where, in the case $p = \infty$, the right hand side of the inequality is intended as $\max \Set{\norm{u}_{F(Y,V)},\norm{w}_{F(Z,V)}}$. \\
            A morphism $\psi \colon (F,c) \rightarrow (G,d)$ in the category $\NVp$ is a natural transformation $\psi \colon F \Rightarrow G$ such that $\psi_{(X,V)} \circ c_{(X,V)} = d_{(X,V)}$, for all $(X,V) \in \opMeas \times \NVect$.
        \end{definition}

        \begin{remark}
        \label{rem:notation for the maps c 1}
            In axiom \ref{ax:compatibility with complementary embeddings in the normed case 1} and throughout the paper, for any object $(F,c) \in \NVp$, we adopt a specific notation to denote the action of embeddings of measure spaces on the assignment $c$, via the bifunctor $F$.
            More precisely, given any morphism $\left( \left( \iota(X),\iota^{-1} \right)^{\mathrm{op}},\id_V \right) \colon (X,V) \rightarrow (Y,V)$ in the category $\opMeas \times \NVect$, where $\iota \colon X \rightarrow Y$ is an embedding, we denote
            \begin{equation*}
                c_{(X,V)}^{(Y,V)} := F_{\left( \left( \iota(X),\iota^{-1} \right)^{\mathrm{op}},\id_V \right)} \circ c_{(X,V)}.
            \end{equation*}
            There is a double abuse in this notation, in the fact that the map embedding $X$ into $Y$ is not immediately evident and the notation is the same as the one introduced in Remark \ref{rem:notation for the maps c}.
            Nevertheless, which embedding of measure spaces is being considered and in which category should always be clear from the context.
        \end{remark}
        
        \begin{example}
        \label{ex:simple functions with values in normed vector spaces 2}
            The fundamental example of an object in the category $\NVp$ is denoted, by abuse of notation, with the same symbol $(S^p,\sigma)$ as in Example \ref{ex:simple functions with values in normed vector spaces 1} and the bifunctor $S^p$ is defined in the same way on objects. \\ 
            The action of $S^p$ on a morphism of the type $\left( \left( \iota(X),\iota^{-1} \right)^{\mathrm{op}},T \right) \colon (X,V) \rightarrow (Y,W)$ in $\opMeas \times \NVect$, where $\iota \colon X \rightarrow Y$ is an embedding, is defined by formula \eqref{eq:action of Sp on emb-morphisms}.
            The fact that $S^p_{\left( \left( \iota(X),\iota^{-1} \right)^{\mathrm{op}},T \right)} \colon S^p(X,V) \rightarrow S^p(Y,W)$ is a contractive map is shown using the same arguments as in Example \ref{ex:simple functions with values in normed vector spaces 1}.
            The action of $S^p$ on morphisms of the type $\left( m^{\mathrm{op}},T \right) \colon (X,V) \rightarrow (Y,W)$ in $\opMeas \times \NVect$, where $m \colon Y \rightarrow X$ is a measure-preserving map, is defined informally by pre-composition with the measure-preserving map $m$ and post-composition with the linear contractive map $T$. 
            More explicitly, we set:
            \begin{equation}
            \label{eq:action of Sp on measure-preserving maps}
                \begin{split}
                    S^p_{\left( m^{\mathrm{op}},T \right)} \colon S^p(X,V) &\rightarrow S^p(Y,W) \\
                    f &\mapsto S^p_{\left( m^{\mathrm{op}},T \right)}f = T \circ f \circ m,
                \end{split}
            \end{equation}
            and well-posedness is guaranteed by the fact that $m$ is a measure-preserving map. 
            The fact that $S^p_{\left( m^{\mathrm{op}},T \right)}$ is a contractive map can be checked, for $p \in [1,\infty)$ and for any $f \in S^p(X,V)$, as follows:
            \begin{align}
                \norm{S^p_{\left( m^{\mathrm{op}},T \right)}f}_{S^p(Y,W)} &= \left( \sum_{v \in V} \norm{T(v)}_W^p \mu_Y\left( m^{-1}\left( f^{-1}(v) \right) \right) \right)^{\frac{1}{p}} \notag \\
                &= \left( \sum_{v \in V} \norm{T(v)}_W^p \mu_X\left( f^{-1}(v) \right) \right)^{\frac{1}{p}} \label{eq:using that m is a measure-preserving map} \\
                & \leq \left( \sum_{v \in V} \norm{v}_V^p \mu_X\left( f^{-1}(v) \right) \right)^{\frac{1}{p}} \label{eq:using contractiveness of T} \\
                &= \norm{f}_{S^p(X,V)}, \notag
            \end{align}
            where equation \eqref{eq:using that m is a measure-preserving map} follows from the fact that $m \colon Y \rightarrow X$ is measure-preserving, while inequality \eqref{eq:using contractiveness of T} follows from the fact that $T \colon V \rightarrow W$ is a contractive map.
            In the case $p = \infty$, we have:
            \begin{align}
                \norm{S^\infty_{\left( m^{\mathrm{op}},T \right)} f}_{S^\infty(Y,W)} &= \inf \Set{\epsilon \geq 0 | \norm{(T \circ f \circ m)(y)}_W \leq \epsilon \quad \text{a.e.~on} ~ Y} \notag \\
                &\leq \inf \Set{\epsilon \geq 0 | \norm{(T \circ f)(x)}_W \leq \epsilon \quad \text{a.e.~on} ~ X} \notag \\
                &\leq \inf \Set{\epsilon \geq 0 | \norm{f(x)}_V \leq \epsilon \quad \text{a.e.~on} ~ X} \label{eq:using again contractiveness of T} \\
                &= \norm{f}_{S^\infty(X,V)}, \notag
            \end{align}
            where inequality \eqref{eq:using again contractiveness of T} follows from the fact that $T \colon V \rightarrow W$ is a contractive map. \\
            For any object $(X,V) \in \opMeas \times \NVect$, the linear map $\sigma_{(X,V)} \colon V \rightarrow S^p(X,V)$ is defined by formula \eqref{eq:definition of sigma} and the fact that it is a bounded map follows from the same arguments used in Example \ref{ex:simple functions with values in normed vector spaces 1}. \\
            Axioms \ref{ax:compatibility with linear contractive maps 1}, \ref{ax:compatibility with complementary embeddings in the normed case 1}, \ref{ax:boundedness of c 1} and \ref{ax:norm inequality wrt complementary embeddings 1} are satisfied for the same reasons, up to notational adaptations regarding the morphisms involved, as in Example \ref{ex:simple functions with values in normed vector spaces 1}.
            Axiom \ref{ax:compatibility with measure-preserving maps in the normed case} is satisfied as a direct consequence of formula \eqref{eq:action of Sp on measure-preserving maps}. 
        \end{example}

        \begin{remark}
        \label{rem:forgetful functor from NVp to NVembp}
            It is straightforward to extend the faithful and injective-on-objects functor $\embMeas \hookrightarrow \opMeas$ outlined in Remark \ref{rem:structure of morphisms in Meas}, to a faithful and injective-on-objects bifunctor $\embMeas \times \NVect \hookrightarrow \opMeas \times \NVect$. 
            In other words, we can regard $\embMeas \times \NVect$ as a subcategory of $\opMeas \times \NVect$. \\
            This induces, in turn, for any $p \in [1,\infty]$, a faithful functor $\mathcal{U}^p \colon \NVp \rightarrow \NVembp$. 
            In particular, the functor $\mathcal{U}^p$ acts, on any object $(F,c) \in \NVp$, by restricting the domain of the bifunctor $F \colon \opMeas \times \NVect \rightarrow \NVect$ to the subcategory $\embMeas \times \NVect$ and leaving the assignment $c$ unchanged.
            The validity of axioms \ref{ax:compatibility with linear contractive maps}, \ref{ax:compatibility with complementary embeddings in the normed case}, \ref{ax:boundedness of c} and \ref{ax:norm inequality wrt complementary embeddings} for $\mathcal{U}^p(F,c)$ is ensured by the fact that $(F,c)$ satisfies axioms \ref{ax:compatibility with linear contractive maps 1}, \ref{ax:compatibility with complementary embeddings in the normed case 1}, \ref{ax:compatibility with measure-preserving maps in the normed case}, \ref{ax:boundedness of c 1} and \ref{ax:norm inequality wrt complementary embeddings 1}.    
            The action of the functor $\mathcal{U}^p$ on any morphism $\psi \colon (F,c) \rightarrow (G,d)$ in $\NVp$ is defined by restricting the naturality of $\psi \colon F \Rightarrow G$ to the morphisms of the category $\embMeas \times \NVect$. \\
            As a special instance of this feature, we have that the object $(S^p,\sigma) \in \NVembp$, discussed in Example \ref{ex:simple functions with values in normed vector spaces 1}, is the image of the object $(S^p,\sigma) \in \NVp$, discussed in Example \ref{ex:simple functions with values in normed vector spaces 2}, via the functor $\mathcal{U}^p$.
            Incidentally, this also motivates why we use the same symbol to denote them.   
        \end{remark}

    \subsection{Preparatory lemmas and main result}
    \label{sec:preparatory lemmas and main result}
    
        In this section, we first show four technical lemmas, which represent the generalisation to our setting, respectively, of the results \cite[Lemma 3.1]{Leinster2023}, \cite[Lemma 3.2]{Leinster2023}, \cite[Lemma 3.3]{Leinster2023} and \cite[Lemma 3.5]{Leinster2023}.
        We then use the lemmas to prove the first main result of this paper.

        \begin{remark}
        \label{rem:morphisms and functors in product categories in the normed case}
            Before proceeding, we point out an elementary observation which will be useful in the sequel.
            The bifunctors belonging to $\NVembp$ and $\NVp$ are defined, respectively, on the categories $\embMeas \times \NVect$ and $\opMeas \times \NVect$. 
            As in any product category, any two morphisms, say $(\iota,T) \colon (X,V) \rightarrow (Y,W)$ in $\embMeas \times \NVect$ and $\left( (D,p)^{\mathrm{op}},T \right) \colon (X,V) \rightarrow (Y,W)$ in $\opMeas \times \NVect$, can be decomposed as:
            \begin{equation} 
            \label{eq:decomposition of morphisms in product categories in the normed case}
                \begin{split}
                    (\iota,T) &= (\iota,\id_W) \circ (\id_X,T) = (\id_Y,T) \circ (\iota,\id_V), \\
                    \left( (D,p)^{\mathrm{op}},T \right) &= \left( (D,p)^{\mathrm{op}},\id_W \right) \circ (\id_X,T) = (\id_Y,T) \circ \left( (D,p)^{\mathrm{op}},\id_V \right).
                \end{split}
            \end{equation}
            The action on these morphisms of any two bifunctors $F \colon \embMeas \times \NVect \rightarrow \mathbf{C}$ and $G \colon \opMeas \times \NVect \rightarrow \mathbf{C}$ to any category $\mathbf{C}$ gives, respectively:
            \begin{equation*}
                \begin{split}
                    F_{(\iota,T)} &= F_{(\iota,\id_W)} \circ F_{(\id_X,T)} = F_{(\id_Y,T)} \circ F_{(\iota,\id_V)}, \\
                    G_{\left( (D,p)^{\mathrm{op}},T \right)} &= G_{\left( (D,p)^{\mathrm{op}},\id_W \right)} \circ G_{(\id_X,T)} = G_{(\id_Y,T)} \circ G_{\left( (D,p)^{\mathrm{op}},\id_V \right)}.
                \end{split}
            \end{equation*}
        \end{remark}

        \begin{lemma}
        \label{lemma:Beck-Chevalley conditions in the normed case}
            Let $F \colon \opMeas \times \NVect \rightarrow \mathbf{C}$ be a bifunctor taking values in any category $\mathbf{C}$. 
            Let $(m^\mathrm{op},T) \colon (X,V) \rightarrow (Y,W)$ be a morphism in $\opMeas \times \NVect$, where $m \colon Y \rightarrow X$ is a measure-preserving map, and let $(B,V)$ be an object in $\opMeas \times \NVect$, where $B \subseteq X$ is a measurable subset. 
            We write the associated commutative diagram in the category $\opMeas \times \NVect$ as:
            \begin{equation}  
            \label{eq:Beck-Chevally diagram in the normed case}
                \begin{tikzcd}
                    (B,V) \ar[dd, "{\left( \left( B,\id_B \right)^{\mathrm{op}},\id _V \right)}"'] \ar[rrr, "{\left( (m')^{\mathrm{op}},T \right)}"] & & & \left( m^{-1}(B),W \right) \ar[dd, "{\left( \left( m^{-1}(B),\id_{m^{-1}(B)} \right)^{\mathrm{op}},\id_W \right)}"]  \\
                        & & & \\
                    (X,V) \ar[rrr, "{\left( m^\mathrm{op},T \right)}"] & & & (Y,W)
                \end{tikzcd}
            \end{equation}
            where $m' \colon m^{-1}(B) \rightarrow B$ denotes the restriction of the measure-preserving map $m$ to $m^{-1}(B)$. 
            Then, the following diagram commutes in $\mathbf{C}$: 
            \begin{equation*}
                \begin{tikzcd}
                    F(B,V) \ar[dd, "F_{\left( \left( B,\id_B \right)^{\mathrm{op}},\id _V \right)}"'] \ar[rrr, "F_{\left( (m')^{\mathrm{op}},T \right)}"] & & & F\left( m^{-1}(B),W \right) \ar[dd, "F_{( ( m^{-1}(B),\id_{m^{-1}(B)} )^{\mathrm{op}},\id_W )}"]  \\
                        & & & \\
                    F(X,V) \ar[rrr, "F_{\left( m^\mathrm{op},T \right)}"] & & & F(Y,W)
                \end{tikzcd}
            \end{equation*}
        \end{lemma}

        \begin{proof}
            The result follows directly from functoriality of $F$ and from commutativity of the following diagram in the category $\Meas \times \NVect$
            \begin{equation*}
                \begin{tikzcd}
                    \left( m^{-1}(B),V \right) \ar[rr, "{(m',T)}"] & & (B,W)  \\
                        & & \\
                    (Y,V) \ar[uu, "{\left( \left( m^{-1}(B),\id_{m^{-1}(B)} \right),\id_V \right)}"] \ar[rr, "{(m,T)}"] & & (X,W) \ar[uu, "{\left( \left( B,\id_B \right),\id _W \right)}"']
                \end{tikzcd}
            \end{equation*}
            where, according to Remark \ref{rem:structure of morphisms in Meas}, the morphisms $\left( m^{-1}(B),\id_{m^{-1}(B)} \right) \colon Y \rightarrow m^{-1}(B)$ and $\left( B,\id_B \right) \colon X \rightarrow B$ in the category $\Meas$ are associated, respectively, to the measurable subsets $m^{-1}(B) \subseteq Y$ and $B \subseteq X$.
        \end{proof}

        \begin{lemma}
        \label{lemma:functors and pairwise disjoint embeddings in the normed case}
            Let $\big\{ (\iota_r,\id_V) \colon (X_r,V) \rightarrow (Y,V) \big\}_{1 \leq r \leq n}$, for some $n \in \N$, be a finite family of morphisms in $\embMeas \times \NVect$, where $\iota_r \colon X_r \rightarrow Y$ are embeddings with pairwise disjoint images.
            Let $(F,c) \in \NVembp$. 
            Then, we have that:
            \begin{equation*}                           
                c_{\left( \iota_1(X_1) \cup \cdots \cup \iota_n(X_n),V \right)}^{(Y,V)} = c_{(X_1,V)}^{(Y,V)} + \cdots + c_{(X_n,V)}^{(Y,V)},
            \end{equation*}
            where, on the left-hand side, the embedding of the subset $\iota_1(X_1) \cup \cdots \cup \iota_n(X_n) \subseteq Y$ is simply given by the inclusion, while, on the right-hand side, we consider the embeddings $\iota_r \colon X_r \rightarrow Y$.
            In particular, $c_{(\emptyset,V)}^{(Y,V)} = 0$, for any $Y \in \embMeas$.
        \end{lemma}

        \begin{proof}
            By abuse of notation, let us denote with $X = \iota_1(X_1) \cup \cdots \cup \iota_n(X_n) \subseteq Y$, for any $n \in \N$. 
            First, we prove, by induction on $n$, that 
            \begin{equation}
            \label{eq:inductive sum for disjoint embeddings}
                c_{(X,V)} = c_{(X_1,V)}^{(X,V)} + \cdots + c_{(X_n,V)}^{(X,V)}.
            \end{equation}
            The basic step, $n=0$, follows applying axiom \ref{ax:compatibility with complementary embeddings in the normed case} to the couple of morphisms $(\emptyset,V) \xrightarrow{~(\id_{\emptyset},\id_V)~} (\emptyset,V) \xleftarrow{~(\id_{\emptyset},\id_V)~} (\emptyset,V)$, which gives $c_{(\emptyset,V)} + c_{(\emptyset,V)} = c_{(\emptyset,V)}$, i.e.~$c_{(\emptyset,V)} = 0$. \\
            The inductive step follows again from axiom \ref{ax:compatibility with complementary embeddings in the normed case}. 
            Indeed, assume formula \eqref{eq:inductive sum for disjoint embeddings} holds for $n \in \N$ and consider a family of morphisms $\{(\iota_r,\id_V) \colon (X_r,V) \rightarrow (Y,V)\}_{1 \leq r \leq n+1}$, where the embeddings $\iota_r \colon X_r \rightarrow Y$, for $1 \leq r \leq n+1$, have pairwise disjoint images.  
            We obtain:
            \begin{equation*}
                \begin{split}
                    c_{(X,V)} &= c_{\left( \iota_1(X_1) \cup \cdots \cup \iota_{n+1}(X_{n+1}),V \right)} \\
                    &= c_{\left( \iota_1(X_1) \cup \cdots \cup \iota_n(X_n),V \right)}^{(X,V)} + c_{(X_{n+1},V)}^{(X,V)} \\
                    &= c_{(X_1,V)}^{(X,V)} + \cdots + c_{(X_n,V)}^{(X,V)} + c_{(X_{n+1},V)}^{(X,V)},
                \end{split}
            \end{equation*}
            where, in the second equality, we used axiom \ref{ax:compatibility with complementary embeddings in the normed case} and, in the third equality, we used the inductive hypothesis. \\ 
            To conclude, functoriality of $F$ with respect to the embedding of $X$ into $Y$ gives:
            \begin{equation*}
                c_{(X,V)}^{(Y,V)} = \left( c_{(X_1,V)}^{(X,V)} + \cdots + c_{(X_n,V)}^{(X,V)} \right)^{(Y,V)} = c_{(X_1,V)}^{(Y,V)} + \cdots + c_{(X_n,V)}^{(Y,V)}.
            \end{equation*}
        \end{proof}

        \begin{lemma}
        \label{lemma:functors and measure-preserving maps in the normed case}
            Let $\left( m^{\mathrm{op}},T \right) \colon (X,V) \rightarrow (Y,W)$ be a morphism in the category $\opMeas \times \NVect$, where $m \colon Y \rightarrow X$ is a measure-preserving map. 
            Let $(B,V)$ be an object in $\opMeas \times \NVect$, where $B \subseteq X$ is a measurable subset. 
            Finally, let $(F,c) \in \NVp$. 
            Then, we have that:
            \begin{equation*}
                F_{\left( m^{\mathrm{op}},T \right)} \circ c_{(B,V)}^{(X,V)} = c_{\left( m^{-1}(B),W \right)}^{(Y,W)} \circ T.
            \end{equation*}
        \end{lemma}

        \begin{proof}
            To obtain the desired result, it suffices to expand and manipulate the left-hand side of the equation above as follows:
            \begin{align}
                F_{\left( m^\mathrm{op},T \right)} \circ c_{(B,V)}^{(X,V)} &= F_{\left( m^\mathrm{op},T \right)} \circ F_{\left( \left( B,\id_B \right)^{\mathrm{op}},\id _V \right)} \circ c_{(B,V)} \notag \\
                &= F_{( ( m^{-1}(B),\id_{m^{-1}(B)} )^{\mathrm{op}},\id_W )} \circ F_{\left( (m')^{\mathrm{op}},T \right)} \circ c_{(B,V)} \label{eq:commutativity of the Beck-Chevalley diagram in the normed case} \\
                &= F_{( ( m^{-1}(B),\id_{m^{-1}(B)} )^{\mathrm{op}},\id_W )} \circ F_{(\id_{m^{-1}(B)},T)} \circ F_{\left( (m')^{\mathrm{op}},\id_V \right)} \circ c_{(B,V)} \label{eq:decomposition of product morphisms in the normed case} \\
                &= F_{( ( m^{-1}(B),\id_{m^{-1}(B)} )^{\mathrm{op}},\id_W )} \circ F_{(\id_{m^{-1}(B)},T)} \circ c_{\left( m^{-1}(B),V \right)} \label{eq:application of axiom II in the normed case} \\
                &= F_{( ( m^{-1}(B),\id_{m^{-1}(B)} )^{\mathrm{op}},\id_W )} \circ c_{\left( m^{-1}(B),W \right)} \circ T \label{eq:application of axiom O in the normed case} \\
                &= c_{\left( m^{-1}(B),W \right)}^{(Y,W)} \circ T \notag.
            \end{align}
            The first equality is simply writing explicitly $c_{(B,V)}^{(X,V)}$.
            Equation \eqref{eq:commutativity of the Beck-Chevalley diagram in the normed case} follows from Lemma \ref{lemma:Beck-Chevalley conditions in the normed case}.
            Equation \eqref{eq:decomposition of product morphisms in the normed case} follows applying Remark \ref{rem:morphisms and functors in product categories in the normed case}.
            Equation \eqref{eq:application of axiom II in the normed case} follows from Axiom \ref{ax:compatibility with measure-preserving maps in the normed case}.
            Equation \eqref{eq:application of axiom O in the normed case} follows from Axiom \ref{ax:compatibility with linear contractive maps 1}.
            Finally, the last equation is simply a matter of notation.  
        \end{proof}
        
        \begin{lemma}
        \label{lemma:compatibility of c with embeddings in the normed case}
            Let $(F,c) \in \NVembp$ and let $(\iota,\id_V) \colon (X,V) \rightarrow (Y,V)$ be a morphism in the category $\embMeas \times \NVect$. 
            Then, the following statements hold true:
            \begin{itemize}
                \item[(i)] for $p \in [1,\infty)$, we have that
                \begin{equation*}
                    \norm{c_{(X,V)}^{(Y,V)}}_{\mathcal{L}(V,F(Y,V))} \leq \mu_Y\left( \iota(X) \right)^\frac{1}{p}; 
                \end{equation*}
                \item[(ii)] for $p = \infty$, we have that 
                \begin{equation*} 
                    \begin{cases}
                        \norm{c_{(X,V)}^{(Y,V)}}_{\mathcal{L}(V,F(Y,V))} = 0   & \text{if} ~ \mu_Y\left( \iota(X) \right) = 0, \\
                        \norm{c_{(X,V)}^{(Y,V)}}_{\mathcal{L}(V,F(Y,V))} \leq 1   & \text{otherwise.}
                    \end{cases}
                \end{equation*}
            \end{itemize}
            In particular, $c_{(X,V)}^{(Y,V)} = 0$ whenever $X$ is a measure-zero subspace of $Y$. 
        \end{lemma}

        \begin{proof}
            First, we observe that
            \begin{equation*}
                \begin{split}
                    \norm{c_{(X,V)}^{(Y,V)}}_{\mathcal{L}(V,F(Y,V))} &= \norm{F_{(\iota,\id_V)} \circ c_{(X,V)}}_{\mathcal{L}(V,F(Y,V))} \\
                    &\leq \norm{c_{(X,V)}}_{\mathcal{L}(V,F(X,V))},
                \end{split}
            \end{equation*}
            because $F_{(\iota,\id_V)}$ is a contractive map. 
            The statement then follows directly from axiom \ref{ax:boundedness of c} and from the fact that $\mu_X(X) = \mu_Y\left( \iota(X) \right)$, since $\iota$ is an embedding. 
        \end{proof}
        
        \begin{theorem} 
        \label{thm:universal properties of simple functions with values in normed vector spaces}
            Let $p \in [1,\infty]$. 
            The bifunctors of the normed vector spaces of equivalence classes of simple functions with values in normed vector spaces, under equality almost everywhere, have the following universal properties:
            \begin{itemize}
                \item[(i)] $(S^p,\sigma)$ is the initial object of the category $\NVembp$;
                \item[(ii)] $(S^p,\sigma)$ is the initial object of the category $\NVp$.
            \end{itemize}
        \end{theorem}

        \begin{proof} 
            First, we consider statement (i). 
            Let $(F,c) \in \NVembp$. 
            We want to show that there exists a unique morphism $\eta \colon (S^p,\sigma) \rightarrow (F,c)$ in the category $\NVembp$.

            \emph{Existence.}\hspace{.5em} We observe that, given an object $(X,V) \in \embMeas \times \NVect$, for any element $f \in S^p(X,V)$ we can choose a representative (denoted, by abuse of notation, with the same symbol $f$) and write:
            \begin{equation}
            \label{eq:expression of simple functions using sigma}
                f = \sum_{v \in V} \sigma_{\left( f^{-1}(v),V \right)}^{(X,V)}(v),
            \end{equation}
            where the sum over $v \in V$ is finite because $f$ is a simple function and because, according to Lemma \ref{lemma:compatibility of c with embeddings in the normed case}, $\sigma_{\left( f^{-1}(v),V \right)}^{(X,V)} = 0$ whenever $f^{-1}(v) \subset X$ is a measure-zero subset.
            We define, then:
            \begin{equation}
            \label{eq:definition of eta}
                \eta_{(X,V)}(f) = \sum_{v \in V} c_{(f^{-1}(v),V)}^{(X,V)}(v) \in F(X,V). 
            \end{equation}
            It is routine to check, combining Lemma \ref{lemma:functors and pairwise disjoint embeddings in the normed case} with Lemma \ref{lemma:compatibility of c with embeddings in the normed case}, that the definition of $\eta_{(X,V)}(f)$ does not depend on the choice of a representative for $f$.

            To show that the map $\eta_{(X,V)} \colon S^p(X,V) \rightarrow F(X,V)$ just defined is linear, let $f,g \in S^p(X,V)$. By formula \eqref{eq:definition of eta}, we have:
            \begin{equation}
            \label{eq:linearity of the morphism from Sp to F}
                \eta_{(X,V)}(f + g) = \sum_{v \in V} c_{\left( (f + g)^{-1}(v),V \right)}^{(X,V)}(v). 
            \end{equation}
            Up to measure-zero subsets of $X$, we can write the following decomposition: 
            \begin{equation*}
                (f + g)^{-1}(v) = \bigsqcup_{\substack{w,z \in V \\ w + z = v}} f^{-1}(w) \cap g^{-1}(z).
            \end{equation*}
            Noting that the family of pairwise disjoint measurable subsets $\left( f^{-1}(w) \cap g^{-1}(z) \right)$, indexed by $w,z \in V$ with $w + z = v$, is finite due to $f$ and $g$ being simple functions, we can apply Lemma \ref{lemma:functors and pairwise disjoint embeddings in the normed case} and obtain:
            \begin{equation*}
                c_{\left( (f + g)^{-1}(v),V \right)}^{(X,V)} = \sum_{\substack{w,z \in V \\ w + z = v}} c_{\left( f^{-1}(w) \cap g^{-1}(z),V \right)}^{(X,V)}.
            \end{equation*}
            Substituting the last expression into formula \eqref{eq:linearity of the morphism from Sp to F}, we get:
            \begin{align}
                \eta_{(X,V)}(f &+ g) = \sum_{v \in V} \sum_{\substack{w,z\in V \\ w + z = v}} c_{\left( f^{-1}(w)\cap g^{-1}(z),V \right)}^{(X,V)}(w + z) \notag \\
                &= \sum_{v \in V} \sum_{\substack{w,z \in V \\ w + z = v}} c_{\left( f^{-1}(w) \cap g^{-1}(z),V \right)}^{(X,V)}(w) + \sum_{v \in V} \sum_{\substack{w,z \in V \\ w + z = v}} c_{\left( f^{-1}(w)\cap g^{-1}(z),V \right)}^{(X,V)}(z) \label{eq:use linearity of the map c} \\
                &= \sum_{w \in V} \sum_{v \in V} c_{\left( f^{-1}(w) \cap g^{-1}(v - w),V \right)}^{(X,V)}(w) + \sum_{z \in V} \sum_{v \in V} c_{\left( f^{-1}(v-z) \cap g^{-1}(z),V \right)}^{(X,V)}(z) \label{eq:rename the indices in the summations} \\
                &= \sum_{w \in V} c_{\left( f^{-1}(w),V \right)}^{(X,V)}(w) + \sum_{z \in V} c_{\left( g^{-1}(z),V \right)}^{(X,V)}(z) \label{eq:sum up one of the summations} \\
                &= \eta_{(X,V)}(f) + \eta_{(X,V)}(g), \notag
            \end{align}
            where equation \eqref{eq:use linearity of the map c} follows from linearity of the map $c_{\left( f^{-1}(w)\cap g^{-1}(z),V \right)}^{(X,V)}$, equation \eqref{eq:rename the indices in the summations} is simply a rearrangement of the summation indices and equation \eqref{eq:sum up one of the summations} is a consequence of Lemma \ref{lemma:functors and pairwise disjoint embeddings in the normed case} applied to the following decompositions
            \begin{equation*}
                f^{-1}(w) = \bigsqcup_{v \in V} f^{-1}(w) \cap g^{-1}(v - w), \quad \qquad g^{-1}(z) = \bigsqcup_{v \in V} f^{-1}(v-z)\cap g^{-1}(z),
            \end{equation*}
            which hold up to measure-zero subsets of $X$.
            Linearity of $\eta_{(X,V)}$ with respect to scalar multiplication follows from the basic observation that 
            \begin{equation*}
                \lambda f = \sum_{v \in V} \sigma_{\left( (\lambda f)^{-1}(v),V \right)}^{(X,V)}(v) = \sum_{v \in V} \sigma_{\left( f^{-1}(v),V \right)}^{(X,V)}(\lambda v),
            \end{equation*}
            for any $\lambda \in \K$, and from linearity of the maps $c_{\left( f^{-1}(v),V \right)}^{(X,V)}$. \\
            In addition, we need to show that, for any object $(X,V) \in \embMeas \times \NVect$, the map $\eta_{(X,V)} \colon S^p(X,V) \rightarrow F(X,V)$ is contractive.
            Indeed, for $p \in [1,\infty)$ and for any $f \in S^p(X,V)$, this can be seen as follows:
            \begin{align}
                \norm{\eta_{(X,V)}(f)}_{F(X,V)} &= \norm{\sum_{v \in V} c_{\left( f^{-1}(v),V \right)}^{(X,V)}(v)}_{F(X,V)} \notag \\
                &\leq \left( \sum_{v \in V} \norm{c_{\left( f^{-1}(v),V \right)}(v)}^p_{F\left( f^{-1}(v),V \right)} \right)^{\frac{1}{p}} \label{eq:apply norm inequality wrt complementary embeddings in the normed case} \\
                &\leq \left( \sum_{v \in V} \norm{v}_V^p \mu_{f^{-1}(v)}\left( f^{-1}(v) \right) \right)^{\frac{1}{p}} \label{eq:apply axiom III} \\
                &= \left( \sum_{v \in V} \norm{v}^p \mu_X\left( f^{-1}(v) \right) \right)^{\frac{1}{p}} \label{eq:apply that embeddings preserve measure} \\
                &= \norm{f}_{S^p(X,V)}, \notag
            \end{align}
            where the first equation corresponds to the definition of $\eta_{(X,V)}(f)$ according to formula \eqref{eq:definition of eta}, inequality \eqref{eq:apply norm inequality wrt complementary embeddings in the normed case} follows by repeatedly applying axiom \ref{ax:norm inequality wrt complementary embeddings}, inequality \eqref{eq:apply axiom III} follows from axiom \ref{ax:boundedness of c} and equality \eqref{eq:apply that embeddings preserve measure} is a consequence of the fact that the embedding $f^{-1}(v) \hookrightarrow X$ preserves measures, namely, $\mu_{f^{-1}(v)}\left( f^{-1}(v) \right) = \mu_X\left( f^{-1}(v) \right)$. \\
            For $p = \infty$, let us first assume that $\mu_X(X) \neq 0$.
            We then have:
            \begin{align}
                \norm{\eta_{(X,V)}(f)}_{F(X,V)} &= \norm{\sum_{v \in V} c_{\left( f^{-1}(v),V \right)}^{(X,V)}(v)}_{F(X,V)} \notag \\
                &\leq \max_{v \in V}\Set{\norm{c_{\left( f^{-1}(v),V \right)}(v)}_{F\left( f^{-1}(v),V \right)}} \label{eq:apply axiom IV repeatedly} \\
                &\leq \max_{v \in V}\Set{\norm{v}_V \norm{c_{\left( f^{-1}(v),V \right)}}_{\mathcal{L}\left( V,F\left( f^{-1}(v),V \right) \right)}} \label{eq:use the operator norm for c} \\
                &\leq \max_{\substack{v \in V \\ \mu_X\left( f^{-1}(v) \right) \neq 0}} \Set{\norm{v}_V} \label{eq:use axiom III} \\
                &= \norm{f}_{S^\infty(X,V)}, \notag
            \end{align}
            where inequality \eqref{eq:apply axiom IV repeatedly} follows by repeatedly applying axiom \ref{ax:norm inequality wrt complementary embeddings}, inequality \eqref{eq:use the operator norm for c} follows by using the operator norms of the maps $c_{\left( f^{-1}(v),V \right)}$ and, finally, inequality \eqref{eq:use axiom III} follows from axiom \ref{ax:boundedness of c}.
            In the case $\mu_X(X) = 0$, we directly obtain:
            \begin{align}
                \norm{\eta_{(X,V)}(f)}_{F(X,V)} &= \norm{\sum_{v \in V} c_{\left( f^{-1}(v),V \right)}^{(X,V)}(v)}_{F(X,V)} \notag \\
                &\leq \sum_{v \in V} \norm{v}_V \norm{c_{\left( f^{-1}(v),V \right)}^{(X,V)}}_{\mathcal{L}(V,F(X,V))} \notag \\
                &= 0 \notag \\
                &= \norm{f}_{S^\infty(X,V)}, \notag
            \end{align} 
            where $\norm{c_{\left( f^{-1}(v),V \right)}^{(X,V)}}_{\mathcal{L}(V,F(X,V))} = 0$, for all $v \in V$, due to Lemma \ref{lemma:compatibility of c with embeddings in the normed case}.

            Next, we show that $\eta$ defines a natural transformation from the bifunctor $S^p$ to the bifunctor $F$. 
            In other words, we have to show that, for any morphism $(\iota,T) \colon (X,V) \rightarrow (Y,W)$ in $\embMeas \times \NVect$, the diagram
            \begin{equation*}
                \begin{tikzcd}
                    S^p(X,V) \ar[rr, "S^p_{(\iota,T)}"] \ar[dd, "\eta_{(X,V)}"']  &   & S^p(Y,W) \ar[dd, "\eta_{(Y,W)}"] \\
                        &   & \\
                    F(X,V) \ar[rr, "F_{(\iota,T)}"]   &   & F(Y,W)
                \end{tikzcd}
            \end{equation*}
            commutes in the category $\NVect$.
            By formula \eqref{eq:expression of simple functions using sigma}, it suffices to check commutativity when composing the arrows of the diagram with the maps $\sigma_{(B,V)}^{(X,V)}$, where $B \subseteq X$ is any measurable subset. 
            On the one hand, we have:
            \begin{align}
                F_{(\iota,T)} \circ \eta_{(X,V)} \circ \sigma_{(B,V)}^{(X,V)} &= F_{(\iota,T)} \circ c_{(B,V)}^{(X,V)} \notag \\
                &= F_{(\iota,\id_W)} \circ F_{(\id_X,T)} \circ c_{(B,V)}^{(X,V)} \label{eq:separate the two components of the morphism} \\
                &= F_{(\iota,\id_W)} \circ c_{(B,V)}^{(X,V)} \circ T \label{eq:apply axiom O} \\
                &= c_{(B,V)}^{(Y,W)} \circ T, \notag
            \end{align}
            where the first equality follows from the definition of $\eta_{(X,V)}$, see formula \eqref{eq:definition of eta}, in equality \eqref{eq:separate the two components of the morphism} we used Remark \ref{rem:morphisms and functors in product categories in the normed case}, equation \eqref{eq:apply axiom O} follows from axiom \ref{ax:compatibility with linear contractive maps} and the last equality is just a matter of notation.
            On the other hand, using the same arguments, we obtain:
            \begin{align*}
                \eta_{(Y,W)} \circ S^p_{(\iota,T)} \circ \sigma_{(B,V)}^{(X,V)} &= \eta_{(Y,W)} \circ S^p_{(\iota,\id_W)} \circ S^p_{(\id_X,T)} \circ \sigma_{(B,V)}^{(X,V)} \\
                &= \eta_{(Y,W)} \circ S^p_{(\iota,\id_W)} \circ \sigma_{(B,W)}^{(X,W)} \circ T \\
                &= \eta_{(Y,W)} \circ \sigma_{(B,W)}^{(Y,W)} \circ T \\
                &= c_{(B,W)}^{(Y,W)} \circ T.
            \end{align*}
            We have thus shown that $\eta$ defines a natural transformation $S^p \Rightarrow F$ and, moreover, for all $(X,V) \in \embMeas \times \NVect$, we have that $\eta_{(X,V)} \circ \sigma_{(X,V)} = c_{(X,V)}$ directly from formula \eqref{eq:definition of eta}. 
            Hence $\eta \colon (S^p,\sigma) \rightarrow (F,c)$ is a morphism in the category $\NVembp$.

            \emph{Uniqueness.}\hspace{.5em} Assume the existence of a morphism $\eta \colon (S^p,\sigma) \rightarrow (F,c)$ in the category $\NVembp$. 
            For any couple of objects $(X,V),(Y,V) \in \embMeas \times \NVect$, where $X \subseteq Y$ is any measurable subset, naturality of $\eta$ with respect to the morphism $(\iota,\id_V) \colon (X,V) \rightarrow (Y,V)$, with $\iota \colon X \rightarrow Y$ denoting the embedding of $X$ into $Y$, gives the following commutative diagram in $\NVect$
            \begin{equation*}
                \begin{tikzcd}
                    S^p(X,V) \ar[rr, "S^p_{(\iota,\id_V)}"] \ar[dd, "\eta_{(X,V)}"']  &   & S^p(Y,V) \ar[dd, "\eta_{(Y,V)}"] \\
                        &   & \\
                    F(X,V) \ar[rr, "F_{(\iota,\id_V)}"]   &   & F(Y,V)
                \end{tikzcd}
            \end{equation*}
            Composing the arrows of this diagram with the map $\sigma_{(X,V)} \colon V \rightarrow S^p(X,V)$ and exploiting the commutativity, we obtain:
            \begin{align}
                \eta_{(Y,V)} \circ \sigma_{(X,V)}^{(Y,V)} &= \eta_{(Y,V)} \circ S^p_{(\iota,\id_V)} \circ \sigma_{(X,V)} \notag \\
                &= F_{(\iota,\id_V)} \circ \eta_{(X,V)} \circ \sigma_{(X,V)} \notag \\
                &= F_{(\iota,\id_V)} \circ c_{(X,V)} \label{eq:apply the definition of morphism in NVembp} \\
                &= c_{(X,V)}^{(Y,V)}. \notag
            \end{align}
            This implies that $\eta_{(Y,V)}$ is uniquely determined on the image of the maps $\sigma_{(X,V)}^{(Y,V)}$, for any measurable subset $X \subseteq Y$. 
            Hence, by formula \eqref{eq:expression of simple functions using sigma} and linearity, $\eta_{(Y,V)}$ is uniquely determined on all of $S^p(Y,V)$. 
            This completes the proof of statement (i).

            To prove statement (ii), let $(F,c) \in \NVp$. 
            According to Remark \ref{rem:forgetful functor from NVp to NVembp}, we can regard $(F,c)$ as an object in the category $\NVembp$ and, from statement (i), we know that there exists a unique morphism $\eta \colon (S^p,\sigma) \rightarrow (F,c)$ in $\NVembp$. 
            To conclude, it suffices to show that $\eta$ naturally defines a morphism in the category $\NVp$. \\ 
            This amounts to showing that, for any morphism $\left( m^{\mathrm{op}},T \right) \colon (X,V) \rightarrow (Y,W) $ in $\opMeas \times \NVect$, where $m \colon Y \rightarrow X$ is a measure-preserving map, the diagram
            \begin{equation*}
                \begin{tikzcd}
                    S^p(X,V) \ar[rr, "S^p_{\left( m^{\mathrm{op}},T \right)}"] \ar[dd, "\eta_{(X,V)}"']  &   & S^p(Y,W) \ar[dd, "\eta_{(Y,W)}"] \\
                        &   & \\
                    F(X,V) \ar[rr, "F_{\left( m^{\mathrm{op}},T \right)}"]    &   & F(Y,W)
                \end{tikzcd}
            \end{equation*}
            commutes in the category $\NVect$.
            By formula \eqref{eq:expression of simple functions using sigma}, it suffices to check commutativity when composing the arrows of the diagram with the maps $\sigma_{(B,V)}^{(X,V)}$, where $B \subseteq X$ is any measurable subset. 
            Using Lemma \ref{lemma:functors and measure-preserving maps in the normed case}, we have, on the one hand:
            \begin{align*}
                \eta_{(Y,W)} \circ S^p_{\left( m^{\mathrm{op}},T \right) }\circ \sigma_{(B,V)}^{(X,V)} &= \eta_{(Y,W)} \circ \sigma_{\left( m^{-1}(B),W \right)}^{(Y,W)} \circ T \\
                &= c_{\left( m^{-1}(B),W \right)}^{(Y,W)} \circ T.
            \end{align*}
            Using again Lemma \ref{lemma:functors and measure-preserving maps in the normed case}, we obtain, on the other hand:
            \begin{align*}
                F_{\left( m^{\mathrm{op}},T \right)} \circ \eta_{(X,V)}\circ \sigma_{(B,V)}^{(X,V)} &= F_{\left( m^{\mathrm{op}},T \right)} \circ c_{(B,V)}^{(X,V)} \\
                &= c_{\left( m^{-1}(B),W \right)}^{(Y,W)} \circ T.
            \end{align*}
            This completes the proof of the theorem. 
        \end{proof}

        \begin{remark}
        \label{rem:connection to Leinster}
            We conclude this section with a comment on the relation between our approach and the approach of \cite{Leinster2023}.
            Despite the fact that our framework provides, as a particular case, a universal characterisation of the normed vector spaces of equivalence classes of $\K$-valued simple functions, under equality almost everywhere, there does not appear to be any simple relation between the categories $\NVembp$ or $\NVp$ and the categories $\mathcal{N}_{\mathrm{emb}}^p \subset \NVect^{\embMeas}$ or $\mathcal{N}^p \subset \NVect^{\opMeas}$, introduced in \cite[\S3]{Leinster2023}.  
        \end{remark}

\section{From simple functions to \texorpdfstring{$L^p$}{Lp} Bochner integrable functions}
\label{sec:completion}

    We consider now the second step of the characterisation of the Banach spaces of equivalence classes of $L^p$ Bochner integrable functions, under equality almost everywhere.
    Namely, we discuss how to encode the process of completing the normed vector spaces of equivalence classes of simple functions with values in normed vector spaces, under equality almost everywhere, in our categorical framework.
    
    \subsection{Basic definitions, examples and remarks}
    \label{sec:basic definitions examples and remarks 2}

        Once again, we begin by introducing the basic definitions needed in order to spell out the appropriate categorical framework.
        In addition, we collect a number of useful remarks about some non-trivial relations between the various categories involved.
        
        \begin{definition}
            We denote by $\Ban$ the category whose objects are Banach spaces over the field $\K$ (either $\R$ or $\C$) and whose morphisms are linear contractive maps.
            We also introduce the category $\BAN$, whose objects are Banach spaces over the field $\K$ (either $\R$ or $\C$) and whose morphisms are bounded linear maps.
        \end{definition}

        \begin{remark}
            Clearly, $\Ban$ is a subcategory of $\BAN$.   
        \end{remark}

        \begin{remark}
        \label{rem:relating Ban and NVect}
            We point out some basic facts relating the category $\BAN$ to the category $\NVECT$, which will be useful in the sequel.
            First, there is an obvious faithful and injective-on-objects functor $\mathcal{I} \colon \BAN \hookrightarrow \NVECT$ which forgets the completeness of Banach spaces.
            The restriction of this functor to the subcategory $\Ban \subset \BAN$ naturally yields a faithful and injective-on-objects functor denoted, by abuse of notation, with the same symbol $\mathcal{I} \colon \Ban \hookrightarrow \NVect$. \\
            The functor $\mathcal{I}$ has a left adjoint given by the functor $\overline{\cdot} \colon \NVECT \rightarrow \BAN$, which acts on normed vector spaces by taking their completion and assigns to every bounded linear map between normed vector spaces the extension by continuity to the completion of its domain and codomain.
            Since extension by continuity preserves the operator norm of bounded linear maps, the restriction of the completion functor to the subcategory $\NVect \subset \NVECT$ yields a functor denoted, by abuse of notation, with the same symbol $\overline{\cdot} \colon \NVect \rightarrow \Ban$, which is the left adjoint to the functor $\mathcal{I} \colon \Ban \hookrightarrow \NVect$. \\
            Secondly, using the functor $\mathcal{I}$, we can define two faithful and injective-on-objects bifunctors which, by abuse of notation, we denote with the same symbol: 
            \begin{equation*}
                \begin{split}
                    (\id,\mathcal{I}) &\colon \embMeas \times \BAN \hookrightarrow \embMeas \times \NVECT, \\
                    (\id,\mathcal{I}) &\colon \opMeas \times \BAN \hookrightarrow \opMeas \times \NVECT.
                \end{split}
            \end{equation*}
            Clearly, these two bifunctors admit left adjoints given, respectively, by:
            \begin{equation*}
                \begin{split}
                    (\id,\overline{\cdot}) &\colon \embMeas \times \NVECT \rightarrow \embMeas \times \BAN, \\
                    (\id,\overline{\cdot}) &\colon \opMeas \times \NVECT \rightarrow \opMeas \times \BAN.
                \end{split}
            \end{equation*}
            Finally, all the bifunctors above can be restricted, respectively, to the subcategories $\embMeas \times \Ban$, $\opMeas \times \Ban$, $\embMeas \times \NVect$ and $\opMeas \times \NVect$ to give corresponding pairs of adjoint bifunctors.
        \end{remark}
        
        \begin{definition} 
        \label{def:Bemb}
            Let $p \in [1,\infty]$. 
            We denote by $\Bembp$ the category whose objects are pairs $(F,c)$, where $F \colon \embMeas \times \Ban \rightarrow \Ban$ is a bifunctor and $c$ denotes the assignment of a bounded linear (i.e.~continuous) map $c_{(X,V)} \colon V \rightarrow F(X,V)$ to every object $(X,V) \in \embMeas \times \Ban$. 
            In addition, we require axioms \ref{ax:compatibility with linear contractive maps}, \ref{ax:compatibility with complementary embeddings in the normed case}, \ref{ax:boundedness of c} and \ref{ax:norm inequality wrt complementary embeddings} from Definition \ref{def:NV_embp} to hold. \\
            A morphism $\psi \colon (F,c) \rightarrow (G,d)$ in the category $\Bembp$ is a natural transformation $\psi \colon F \Rightarrow G$ such that $\psi_{(X,V)} \circ c_{(X,V)} = d_{(X,V)}$, for all $(X,V) \in \embMeas \times \Ban$.
        \end{definition}

        \begin{example}
        \label{ex:Lpemb functor}
            The fundamental example of an object in the category $\Bembp$ is the pair $(L^p,\sigma)$.
            Here, $L^p \colon \embMeas \times \Ban \rightarrow \Ban$ is the bifunctor which associates, to any object $(X,V) \in \embMeas \times \Ban$, the Banach space $L^p(X,V)$ of equivalence classes of $L^p$ Bochner integrable functions on $X$ with values in $V$, under equality almost everywhere. 
            We recall that the norm on $L^p(X,V)$ is defined, for any $p \in [1,\infty)$ and for any $f \in L^p(X,V)$, as
            \begin{equation}
            \label{eq:Lp norm}
                \norm{f}_{L^p(X,V)} = \left( \int_X \norm{f(x)}_V^p d\mu_{X}(x) \right)^{\frac{1}{p}},
            \end{equation}
            while, for $p = \infty$ and for any $f \in L^p(X,V)$, it is defined as
            \begin{equation}
            \label{eq:Linfty norm}
                \norm{f}_{L^\infty(X,V)} = \inf \Set{\epsilon \geq 0 | \norm{f(x)}_V \leq \epsilon \quad \text{a.e.~on} ~ X}.
            \end{equation}
            The action of the bifunctor $L^p$ on a morphism $(\iota,T) \colon (X,V) \rightarrow (Y,W)$ in the category $\embMeas \times \Ban$ is defined, in an analogous way as in Example \ref{ex:simple functions with values in normed vector spaces 1}, by formula \eqref{eq:action of Sp on emb-morphisms}.
            That $L^p_{(\iota,T)} \colon L^p(X,V) \rightarrow L^p(Y,W)$ is a contractive map can be checked, for any $p \in [1,\infty)$ and for any $f \in L^p(X,V)$, as follows:
            \begin{align}
                \norm{L^p_{(\iota,T)} f}_{L^p(Y,W)} &= \left( \int_{\iota(X)} \norm{(T \circ f \circ \iota^{-1})(y)}_W^p d\mu_{Y}(y) \right)^{\frac{1}{p}} \notag \\
                &= \left( \int_{X} \norm{(T \circ f)(x)}_W^p d\mu_{X}(x) \right)^{\frac{1}{p}} \label{eq:embeddings preserve measures} \\
                &\leq \left( \int_{X} \norm{f(x)}_V^p d\mu_{X}(x) \right)^{\frac{1}{p}} \label{eq:T is a contractive map} \\
                &= \norm{f}_{L^p(X,V)}, \notag
            \end{align}
            where equation \eqref{eq:embeddings preserve measures} follows from the fact that embeddings preserve measures, while inequality \eqref{eq:T is a contractive map} follows from the fact that $T \colon V \rightarrow W$ is a contractive map.
            In the case $p = \infty$, for any $f \in L^\infty(X,V)$, we have:
            \begin{align}
                \norm{L^\infty_{(\iota,T)} f}_{L^\infty(Y,W)} &= \inf \Set{\epsilon \geq 0 | \norm{(T \circ f \circ \iota^{-1})(y)}_W \leq \epsilon \quad \text{a.e.~on} ~ \iota(X)} \notag \\
                &= \inf \Set{\epsilon \geq 0 | \norm{(T \circ f)(x)}_W \leq \epsilon \quad \text{a.e.~on} ~ X} \notag \\
                &\leq \inf \Set{\epsilon \geq 0 | \norm{f(x)}_V \leq \epsilon \quad \text{a.e.~on} ~ X} \label{eq:use again that T is a contractive map} \\
                &= \norm{f}_{L^\infty(X,V)}, \notag
            \end{align}
            where, again, inequality \eqref{eq:use again that T is a contractive map} follows from the fact that $T$ is a contractive map. \\
            For any object $(X,V) \in \embMeas \times \Ban$, the linear map $\sigma_{(X,V)} \colon V \rightarrow L^p(X,V)$ is defined by formula \eqref{eq:definition of sigma}. 
            That $\sigma_{(X,V)}$ is a bounded linear map follows from the same arguments used in Example \ref{ex:simple functions with values in normed vector spaces 1}. \\
            Axioms \eqref{ax:compatibility with linear contractive maps}, \eqref{ax:compatibility with complementary embeddings in the normed case} and \eqref{ax:boundedness of c}, which involve only the assignment $\sigma$, are satisfied for the same reasons as in Example \ref{ex:simple functions with values in normed vector spaces 1}.
            Regarding the validity of axiom \eqref{ax:norm inequality wrt complementary embeddings}, for any $p \in [1,\infty)$ and for any elements $f \in L^p{(Y,V)}$ and $g \in L^p{(Z,V)}$, denoting 
            \begin{equation*}
                h = L^p_{(\alpha,\id_V)}f + L^p_{(\beta,\id_V)}g \in L^p{(X,V)},
            \end{equation*}
            with $Y \xrightarrow{~\alpha~} X \xleftarrow{~\beta~} Z$ any couple of complementary embeddings, we obtain indeed:
            \begin{align}
                \norm{h}_{L^p(X,V)} &= \left( \int_{X} \norm{h(x)}_V^p d\mu_{X}(x) \right)^{\frac{1}{p}} \notag \\
                &= \left( \int_{\alpha(Y)} \norm{h(x)}_V^p d\mu_{X}(x) + \int_{\beta(Z)} \norm{h(x)}_V^p d\mu_{X}(x) \right)^{\frac{1}{p}} \label{eq:use additivity of the Bochner integral} \\
                &= \left( \int_{Y} \norm{f(y)}_V^p d\mu_{Y}(y) + \int_{Z} \norm{g(z)}_V^p d\mu_{Z}(z) \right)^{\frac{1}{p}} \label{eq:use that h corresponds to f on Y and to g on Z and that embeddings preserve measures} \\
                &= \left( \norm{f}_{L^p(Y,V)}^p + \norm{g}_{L^p(Z,V)}^p \right)^{\frac{1}{p}}, \notag
            \end{align}
            where equation \eqref{eq:use additivity of the Bochner integral} follows from the fact that $X = \alpha(Y) \cup \beta(Z)$ and $\alpha(Y) \cap \beta(Z) = \emptyset$ (see Definition \ref{def:complementary embeddings}) and from additivity of the Lebesgue integral, while equation \eqref{eq:use that h corresponds to f on Y and to g on Z and that embeddings preserve measures} follows from the fact that $h$ coincides, respectively, with $f$ on $\alpha(Y)$ and with $g$ on $\beta(Z)$ and from the fact that embeddings preserve measures.
            For $p = \infty$ and for any elements $f \in L^\infty{(Y,V)}$ and $g \in L^\infty{(Z,V)}$, denoting 
            \begin{equation*}
                h = L^\infty_{(\alpha,\id_V)}f + L^\infty_{(\beta,\id_V)}g \in L^\infty{(X,V)},
            \end{equation*}
            with $Y \xrightarrow{~\alpha~} X \xleftarrow{~\beta~} Z$ any couple of complementary embeddings, we obtain:
            \begin{align}
                \norm{h}_{S^\infty(X,V)} &= \inf \Set{\epsilon \geq 0 | \norm{h(x)}_V \leq \epsilon \quad \text{a.e.~on} ~ X} \notag \\
                &= \max \big\{ \inf \Set{\epsilon_1 \geq 0 ~|~ \norm{h(x)}_V \leq \epsilon_1 \quad \text{a.e.~on} ~ \alpha(Y)}, \notag \\
                &\qquad \qquad \inf \Set{\epsilon_2 \geq 0 ~|~ \norm{h(x)}_V \leq \epsilon_2 \quad \text{a.e.~on} ~ \beta(Z)} \big\} \notag \\
                &= \max \big\{ \inf \Set{\epsilon_1 \geq 0 ~|~ \norm{f(y)}_V \leq \epsilon_1 \quad \text{a.e.~on} ~ Y}, \label{eq:use again that h corresponds to f on Y and to g on Z and that embeddings preserve measures} \\
                &\qquad \qquad \inf \Set{\epsilon_2 \geq 0 ~|~ \norm{g(z)}_V \leq \epsilon_2 \quad \text{a.e.~on} ~ Z} \big\} \notag \\
                &= \max \Set{\norm{f}_{S^\infty(Y,V)},\norm{g}_{S^\infty(Z,V)}}, \notag
            \end{align}
            where equation \eqref{eq:use again that h corresponds to f on Y and to g on Z and that embeddings preserve measures} follows from the fact that $h$ coincides, respectively, with $f$ on $\alpha(Y)$ and with $g$ on $\beta(Z)$. 
        \end{example}
        
        \begin{definition}
        \label{def:Bp}
            Let $p \in [1,\infty]$. 
            We denote by $\Bp$ the category whose objects are pairs $(F,c)$, where $F \colon \opMeas \times \Ban \rightarrow \Ban$ is a bifunctor and $c$ denotes the assignment of a bounded linear (i.e.~continuous) map $c_{(X,V)} \colon V \rightarrow F(X,V)$ to every object $(X,V) \in \opMeas \times \Ban$. 
            In addition, we require axioms \ref{ax:compatibility with linear contractive maps 1}, \ref{ax:compatibility with complementary embeddings in the normed case 1}, \ref{ax:compatibility with measure-preserving maps in the normed case}, \ref{ax:boundedness of c 1} and \ref{ax:norm inequality wrt complementary embeddings 1} from Definition \ref{def:NVp} to hold. \\
            A morphism $\psi \colon (F,c) \rightarrow (G,d)$ in the category $\Bp$ is a natural transformation $\psi \colon F \Rightarrow G$ such that $\psi_{(X,V)} \circ c_{(X,V)} = d_{(X,V)}$, for all $(X,V) \in \opMeas \times \Ban$.
        \end{definition}

        \begin{example}
        \label{ex:Lp functor}
            The fundamental example of an object in the category $\Bp$ is the pair $(L^p,\sigma)$.
            Here, $L^p \colon \opMeas \times \Ban \rightarrow \Ban$ is the bifunctor which associates, to any object $(X,V) \in \opMeas \times \Ban$, the Banach space $L^p(X,V)$ of equivalence classes of $L^p$ Bochner integrable functions on $X$ with values in $V$, under equality almost everywhere. 
            The norm on the space $L^p(X,V)$ is defined, for any $p \in [1,\infty)$, by formula \eqref{eq:Lp norm} and, for $p = \infty$, by formula \eqref{eq:Linfty norm}. \\
            The action of the bifunctor $L^p$ on a morphism $\left( \left( \iota(X),\iota^{-1} \right)^{\mathrm{op}},T \right) \colon (X,V) \rightarrow (Y,W)$ in $\opMeas \times \Ban$, where $\iota \colon X \rightarrow Y$ is an embedding, is defined by formula \eqref{eq:action of Sp on emb-morphisms}.
            The fact that $L^p_{\left( \left( \iota(X),\iota^{-1} \right)^{\mathrm{op}},T \right)} \colon L^p(X,V) \rightarrow L^p(Y,W)$ is a contractive map is shown using the same arguments as in Example \ref{ex:Lpemb functor}.
            The action of $L^p$ on morphisms of the type $\left( m^{\mathrm{op}},T \right) \colon (X,V) \rightarrow (Y,W)$ in $\opMeas \times \Ban$, where $m \colon Y \rightarrow X$ is a measure-preserving map, is defined by formula \eqref{eq:action of Sp on measure-preserving maps}.
            The fact that $L^p_{\left( m^{\mathrm{op}},T \right)} \colon L^p(X,V) \rightarrow L^p(Y,W)$ is a contractive map can be verified, for any $p \in [1,\infty)$ and for any $f \in L^p(X,V)$, in the following way:
            \begin{align}
                \norm{L^p_{\left( m^{\mathrm{op}},T \right)} f}_{L^p(Y,W)} &= \left( \int_{Y} \norm{(T \circ f \circ m)(y)}_W^p d\mu_{Y}(y) \right)^{\frac{1}{p}} \notag \\
                &= \left( \int_{X} \norm{(T \circ f)(x)}_W^p d\mu_{X}(x) \right)^{\frac{1}{p}} \label{eq:use that m is a measure-preserving map} \\
                &\leq \left( \int_{X} \norm{f(x)}_V^p d\mu_{X}(x) \right)^{\frac{1}{p}} \label{eq:use once more that T is a contractive map} \\
                &= \norm{f}_{L^p(X,V)}, \notag
            \end{align}
            where equality \eqref{eq:use that m is a measure-preserving map} follows from the fact that $m$ is a measure-preserving map, while inequality \eqref{eq:use once more that T is a contractive map} follows from the fact that $T \colon V \rightarrow W$ is a contractive map.
            In the case $p = \infty$, for any $f \in L^\infty(X,V)$, we have:
            \begin{align}
                \norm{L^\infty_{\left( m^{\mathrm{op}},T \right)} f}_{L^\infty(Y,W)} &= \inf \Set{\epsilon \geq 0 | \norm{(T \circ f \circ m)(y)}_W \leq \epsilon \quad \text{a.e.~on} ~ Y} \notag \\
                &\leq \inf \Set{\epsilon \geq 0 | \norm{(T \circ f)(x)}_W \leq \epsilon \quad \text{a.e.~on} ~ X} \notag \\
                &\leq \inf \Set{\epsilon \geq 0 | \norm{f(x)}_V \leq \epsilon \quad \text{a.e.~on} ~ X} \label{eq:use once more again that T is a contractive map} \\
                &= \norm{f}_{L^\infty(X,V)}, \notag
            \end{align}
            where inequality \eqref{eq:use once more again that T is a contractive map} follows from the fact that $T$ is contractive. \\
            For any object $(X,V) \in \opMeas \times \Ban$, the linear map $\sigma_{(X,V)} \colon V \rightarrow L^p(X,V)$ is defined by formula \eqref{eq:definition of sigma} and its boundedness follows from the same arguments used in Example \ref{ex:simple functions with values in normed vector spaces 1}. \\
            Finally, axioms \eqref{ax:compatibility with linear contractive maps 1}, \eqref{ax:compatibility with complementary embeddings in the normed case 1}, \ref{ax:compatibility with measure-preserving maps in the normed case}, \eqref{ax:boundedness of c 1} and \eqref{ax:norm inequality wrt complementary embeddings 1} are satisfied for the same reasons as in Example \ref{ex:simple functions with values in normed vector spaces 2}.
        \end{example}
        
        \begin{remark}
        \label{rem:on all the forgetful functors}
            Expanding on Remark \ref{rem:relating Ban and NVect}, we point out the existence of two faithful and injective-on-objects functors $\mathcal{V}_{\mathrm{emb}}^p \colon \Bembp \hookrightarrow \NVembp$ and $\mathcal{V}^p \colon \Bp \hookrightarrow \NVp$. \\
            For any object $(F,c) \in \Bembp$ -- where we recall that $F \colon \embMeas \times \Ban \rightarrow \Ban$ is a bifunctor and $c$ assigns, to every object $(X,V) \in \embMeas \times \Ban$, a bounded linear map $c_{(X,V)} \colon V \rightarrow F(X,V)$ -- we define $\mathcal{V}_{\mathrm{emb}}^p(F,c) = \left( \mathcal{V}_{\mathrm{emb}}^pF,\mathcal{V}_{\mathrm{emb}}^pc \right) \in \NVembp$ as follows.
            The bifunctor $\mathcal{V}_{\mathrm{emb}}^pF$ is defined by pre-composing the bifunctor $F$ with the bifunctor $(\id,\overline{\cdot})$ and post-composing it with the functor $\mathcal{I}$ (both were introduced in Remark \ref{rem:relating Ban and NVect}). More explicitly, we set:  
            \begin{equation}
            \label{eq:definition of the functor Vembp}
                \mathcal{V}_{\mathrm{emb}}^pF = \mathcal{I} \circ F \circ (\id,\overline{\cdot}) \colon \embMeas \times \NVect \rightarrow \NVect.
            \end{equation}
            The bounded linear map $\left( \mathcal{V}_{\mathrm{emb}}^pc \right)_{(Y,W)} \colon W \rightarrow \left( \mathcal{V}_{\mathrm{emb}}^pF \right)(Y,W)$ is defined, for any object $(Y,W) \in \embMeas \times \NVect$, as: 
            \begin{equation*}
                \left( \mathcal{V}_{\mathrm{emb}}^pc \right)_{(Y,W)} = \mathcal{I}_{c_{(\id,\overline{\cdot})(Y,W)}}, 
            \end{equation*}
            where the last formula is understood as the functor $\mathcal{I} \colon \BAN \hookrightarrow \NVECT$ acting on the bounded linear map of Banach spaces $c_{(Y,\overline{W})} \colon \overline{W} \rightarrow F\left( Y,\overline{W} \right)$.  
            The validity of axioms \ref{ax:compatibility with linear contractive maps}, \ref{ax:compatibility with complementary embeddings in the normed case}, \ref{ax:boundedness of c} and \ref{ax:norm inequality wrt complementary embeddings} for $\mathcal{V}_{\mathrm{emb}}^p(F,c) \in \NVembp$ is ensured by the fact that $(F,c) \in \Bembp$ satisfies the same axioms. \\
            The action of the functor $\mathcal{V}_{\mathrm{emb}}^p$ on any morphism $\psi \colon (F,c) \rightarrow (G,d)$ in $\Bembp$ is defined by: 
            \begin{equation*}
                \left( \mathcal{V}_{\mathrm{emb}}^p \right)_{\psi} = \mathcal{I} \bullet \psi \bullet (\id,\overline{\cdot}) \colon \mathcal{V}_{\mathrm{emb}}^pF \Rightarrow \mathcal{V}_{\mathrm{emb}}^pG,
            \end{equation*}
            where $\mathcal{I} \bullet \psi \bullet (\id,\overline{\cdot})$ denotes the \emph{whiskered composite} of the natural transformation $\psi$ with the functors $\mathcal{I}$ and $(\id,\overline{\cdot})$ (see \cite[\S 1.7]{Riehl2016} for more details about whiskering).
            The requirement, for any morphism in the category $\NVembp$, that the equality $\left( \left( \mathcal{V}_{\mathrm{emb}}^p \right)_{\psi} \right)_{(Y,W)} \circ \left( \mathcal{V}_{\mathrm{emb}}^pc \right)_{(Y,W)} = \left( \mathcal{V}_{\mathrm{emb}}^pd \right)_{(Y,W)}$ is satisfied for any object $(Y,W) \in \embMeas \times \NVect$, can be checked as follows:
            \begin{align}
                \left( \left( \mathcal{V}_{\mathrm{emb}}^p \right)_{\psi} \right)_{(Y,W)} \circ \left( \mathcal{V}_{\mathrm{emb}}^pc \right)_{(Y,W)} &= \mathcal{I}_{\psi_{(\id,\overline{\cdot})(Y,W)}} \circ \mathcal{I}_{c_{(\id,\overline{\cdot})(Y,W)}} \notag \\
                &= \mathcal{I}_{\psi_{(\id,\overline{\cdot})(Y,W)} \circ c_{(\id,\overline{\cdot})(Y,W)}} \notag \\
                &= \mathcal{I}_{d_{(\id,\overline{\cdot})(Y,W)}} \label{eq:use that psi is a morphism in Bembp} \\
                &= \left( \mathcal{V}_{\mathrm{emb}}^pd \right)_{(Y,W)}, \notag
            \end{align}
            where equation \eqref{eq:use that psi is a morphism in Bembp} follows from the fact that $\psi \colon (F,c) \rightarrow (G,d)$ is a morphism in the category $\Bembp$. \\ 
            The analogous faithful and injective-on-objects functor $\mathcal{V}^p \colon \Bp \hookrightarrow \NVp$ is defined in the exact same way as the functor $\mathcal{V}_{\mathrm{emb}}^p$ except that the category $\embMeas$ has to be substituted everywhere with the category $\opMeas$. \\ 
            Combining $\mathcal{V}_{\mathrm{emb}}^p$ and $\mathcal{V}^p$ with the functor $\mathcal{U}^p \colon \NVp \rightarrow \NVembp$ from Remark \ref{rem:forgetful functor from NVp to NVembp}, we obtain a faithful functor $\tilde{\mathcal{U}}^p \colon \Bp \rightarrow \Bembp$ defined as the restriction of $\mathcal{U}^p$ to the subcategories $\Bp$ and $\Bembp$.
            The overall situation is summarised in the following commutative diagram of functors:
            \begin{equation*}
                \begin{tikzcd}
                    \Bp \ar[rr, "\tilde{\mathcal{U}}^p"] \ar[dd, hook, "\mathcal{V}^p"']    &   & \Bembp \ar[dd, hook, "\mathcal{V}_{\mathrm{emb}}^p"] \\
                        &   & \\
                    \NVp \ar[rr, "\mathcal{U}^p"]     &   & \NVembp
                \end{tikzcd}
            \end{equation*}
            As a special instance of this feature, we have that the object $(L^p,\sigma) \in \Bembp$, discussed in Example \ref{ex:Lpemb functor}, is the image of the object $(L^p,\sigma) \in \Bp$, discussed in Example \ref{ex:Lp functor}, via the functor $\tilde{\mathcal{U}}^p$. 
            Incidentally, this also motivates why we use the same symbol to denote them.
        \end{remark}

    \subsection{Main results}
    \label{sec:main results}

        We are finally ready to prove the main results of this paper. 
        First, we establish the universal properties of the Banach spaces of equivalence classes of $L^p$ Bochner integrable functions, under equality almost everywhere, for all $p \in [1,\infty)$.
        Then, we show how the universal property in the case $p = 1$ naturally yields a unique characterisation of the Bochner integral.  
    
        \begin{theorem}
        \label{thm:universal properties of Bochner integrable functions}
            Let $p \in [1,\infty)$. The bifunctors of the Banach spaces of equivalence classes of $L^p$ Bochner integrable functions, under equality almost everywhere, have the following universal properties:
            \begin{itemize}
                \item[(i)] $(L^p,\sigma)$ is the initial object of the category $\Bembp$;
                \item[(ii)] $(L^p,\sigma)$ is the initial object of the category $\Bp$.
            \end{itemize}        
        \end{theorem}
        
        \begin{proof}
            Let us consider statement (i).
            We define a functor $\mathcal{C}_{\mathrm{emb}}^p \colon \NVembp \rightarrow \Bembp$ as follows. 
            The action of $\mathcal{C}_{\mathrm{emb}}^p$ on any object $(F,c) \in \NVembp$ -- where we recall that $F \colon \embMeas \times \NVect \rightarrow \NVect$ is a bifunctor and $c$ assigns, to every object $(Y,W) \in \embMeas \times \NVect$, a bounded linear map of normed vector spaces $c_{(Y,W)} \colon W \rightarrow F(Y,W)$ -- is denoted by $\mathcal{C}_{\mathrm{emb}}^p(F,c) = \left( \mathcal{C}_{\mathrm{emb}}^pF,\mathcal{C}_{\mathrm{emb}}^pc \right) \in \Bembp$ and is defined informally by pre-composing the bifunctor $F$ with the bifunctor $(\id,\mathcal{I})$ and post-composing it with the functor $\overline{\cdot}$ (both were introduced in Remark \ref{rem:relating Ban and NVect}). More explicitly, we set:
            \begin{equation}
            \label{eq:definition of the functor Cembp}
                \begin{split}
                    \mathcal{C}_{\mathrm{emb}}^pF = \overline{\cdot} \circ F \circ (\id,\mathcal{I}) \colon \embMeas \times \Ban \rightarrow \Ban.
                \end{split}
            \end{equation}
            The bounded linear map $\left( \mathcal{C}_{\mathrm{emb}}^pc \right)_{(X,V)} \colon V \rightarrow \left( \mathcal{C}_{\mathrm{emb}}^pF \right)(X,V)$ is defined, for any object $(X,V) \in \embMeas \times \Ban$, as: 
            \begin{equation}
            \label{eq:action of Cembp on c}
                \left( \mathcal{C}_{\mathrm{emb}}^pc \right)_{(X,V)} = \overline{\cdot}_{c_{(\id,\mathcal{I})(X,V)}}, 
            \end{equation}
            where the completion functor $\overline{\cdot} \colon \NVECT \rightarrow \BAN$ acts on the bounded linear map of normed vector spaces $c_{(X,\mathcal{I}V)} \colon \mathcal{I}V \rightarrow F\left( Y,\mathcal{I}V \right)$.
            The validity of axioms \ref{ax:compatibility with linear contractive maps}, \ref{ax:compatibility with complementary embeddings in the normed case}, \ref{ax:boundedness of c} and \ref{ax:norm inequality wrt complementary embeddings} for $\mathcal{C}_{\mathrm{emb}}^p(F,c) \in \Bembp$ is ensured by the fact that the axioms do not rely on any completeness assumption for the vector spaces and by the fact that the object $(F,c) \in \NVembp$ satisfies the same axioms. \\
            The action of the functor $\mathcal{C}_{\mathrm{emb}}^p$ on any morphism $\psi \colon (F,c) \rightarrow (G,d)$ in the category $\NVembp$ is defined by:
            \begin{equation*}
                \left( \mathcal{C}_{\mathrm{emb}}^p \right)_{\psi} = \overline{\cdot} \bullet \psi \bullet (\id,\mathcal{I}) \colon \mathcal{C}_{\mathrm{emb}}^pF \Rightarrow \mathcal{C}_{\mathrm{emb}}^pG,
            \end{equation*}
            where $\overline{\cdot} \bullet \psi \bullet (\id,\mathcal{I})$ denotes the whiskered composite of the natural transformation $\psi$ with the functors $\overline{\cdot}$ and $(\id,\mathcal{I})$ (see \cite[\S 1.7]{Riehl2016} about whiskering).
            The requirement that the equality $\left( \left( \mathcal{C}_{\mathrm{emb}}^p \right)_{\psi} \right)_{(X,V)} \circ \left( \mathcal{C}_{\mathrm{emb}}^pc \right)_{(X,V)} = \left( \mathcal{C}_{\mathrm{emb}}^pd \right)_{(X,V)}$ is satisfied for any object $(X,V) \in \embMeas \times \Ban$ can be checked as follows:
            \begin{align}
                \left( \left( \mathcal{C}_{\mathrm{emb}}^p \right)_{\psi} \right)_{(X,V)} \circ \left( \mathcal{C}_{\mathrm{emb}}^pc \right)_{(X,V)} &= \overline{\cdot}_{\psi_{(\id,\mathcal{I})(X,V)}} \circ \overline{\cdot}_{c_{(\id,\mathcal{I})(X,V)}} \notag \\
                &= \overline{\cdot}_{\psi_{(\id,\mathcal{I})(X,V)} \circ c_{(\id,\mathcal{I})(X,V)}} \notag \\
                &= \overline{\cdot}_{d_{(\id,\mathcal{I})(X,V)}} \label{eq:use that psi is a morphism in NVembp} \\
                &= \left( \mathcal{C}_{\mathrm{emb}}^pd \right)_{(X,V)}, \notag
            \end{align}
            where equation \eqref{eq:use that psi is a morphism in NVembp} follows from the fact that $\psi \colon (F,c) \rightarrow (G,d)$ is a morphism in the category $\NVembp$.

            We claim that the functor $\mathcal{C}_{\mathrm{emb}}^p \colon \NVembp \rightarrow \Bembp$ is left adjoint to the functor $\mathcal{V}_{\mathrm{emb}}^p \colon \Bembp \hookrightarrow \NVembp$ described in Remark \ref{rem:on all the forgetful functors}.                 
            To verify this, it suffices to notice that, while the functor $\mathcal{V}_{\mathrm{emb}}^p$ is defined, in formula \eqref{eq:definition of the functor Vembp}, by pre-composition and post-composition with certain functors, the functor $\mathcal{C}_{\mathrm{emb}}^p$ is defined, in formula \eqref{eq:definition of the functor Cembp}, by pre-composition and post-composition with the left adjoints of the corresponding functors (cf.~Remark \ref{rem:relating Ban and NVect}).
            The claim then follows from standard results in category theory, see for example \cite[Proposition 4.4.6]{Riehl2016}. \\
            Using the well-known fact that left adjoint functors preserve initial objects, in combination with Theorem \ref{thm:universal properties of simple functions with values in normed vector spaces}, we thus obtain that the initial object of the category $\Bembp$ is $\mathcal{C}_{\mathrm{emb}}^p(S^p,\sigma) = \left( \mathcal{C}_{\mathrm{emb}}^p S^p,\mathcal{C}_{\mathrm{emb}}^p \sigma \right)$, for any $p \in [1,\infty)$. \\
            According to formula \eqref{eq:definition of the functor Cembp}, the bifunctor $\mathcal{C}_{\mathrm{emb}}^p S^p \colon \embMeas \times \Ban \rightarrow \Ban$ associates, to any object $(X,V) \in \embMeas \times \Ban$, the Banach space
            \begin{equation*}
                \begin{split}
                    \left( \mathcal{C}_{\mathrm{emb}}^p S^p \right)(X,V) &= \left( \overline{\cdot} \circ S^p \circ (\id,\mathcal{I}) \right)(X,V) \\
                    &= \left( \overline{\cdot} \circ S^p \right)(X,\mathcal{I}V) \\
                    &= \overline{S^p(X,\mathcal{I}V)} \\
                    &= L^p(X,V), 
                \end{split}
            \end{equation*}
            where the last equality corresponds to the definition of the Banach spaces of equivalence classes of $L^p$ Bochner integrable functions, under equality almost everywhere, as the completion, with respect to the $L^p$ norm \eqref{eq:Lp norm}, of the normed vector spaces of equivalence classes of simple functions, under equality almost everywhere (we refer to \cite{Cohn2013} and \cite{Yosida1995} for more details). \\
            Regarding the bounded linear map $\left( \mathcal{C}_{\mathrm{emb}}^p \sigma \right)_{(X,V)} \colon V \rightarrow \mathcal{C}_{\mathrm{emb}}^p S^p(X,V)$, for any object $(X,V) \in \embMeas \times \Ban$, it is obtained, according to formula \eqref{eq:action of Cembp on c}, as the extension by continuity of the bounded linear map $\sigma_{(X,\mathcal{I}V)} \colon \mathcal{I}V \rightarrow S^p(X,\mathcal{I}V)$ between normed vector spaces to a bounded linear map from the completion of $\mathcal{I}V$ to the completion of $S^p(X,\mathcal{I}V)$.
            Since $\mathcal{I}V$ is already complete, the extension is in fact trivial and the resulting map $\left( \mathcal{C}_{\mathrm{emb}}^p \sigma \right)_{(X,V)}$ coincides with the map $\sigma_{(X,V)}$, except that its values are regarded as elements of $\overline{S^p(X,\mathcal{I}V)} = L^p(X,V)$. \\
            We have thus identified $\mathcal{C}_{\mathrm{emb}}^p(S^p,\sigma) = (L^p,\sigma)$ as the initial object of the category $\Bembp$, for any $p \in [1,\infty)$, and this concludes the proof of statement (i).
            
            The same argument, with the category $\embMeas$ substituted everywhere by the category $\opMeas$, the functors $\mathcal{C}_{\mathrm{emb}}^p$ and $\mathcal{V}_{\mathrm{emb}}^p$ substituted by the functors $\mathcal{C}^p$ and $\mathcal{V}^p$, the categories $\NVembp$ and $\Bembp$ substituted by the categories $\NVp$ and $\Bp$, yields the proof of statement (ii). 
        \end{proof}

        \begin{remark}
        \label{rem:issue with the infinity case}
            The case $p = \infty$ is not covered by Teorem \ref{thm:universal properties of Bochner integrable functions} because the normed vector spaces $S^\infty(X,V)$ of equivalence classes of simple functions, under equality almost everywhere, endowed with the $L^\infty$ norm \eqref{eq:Linfty norm} are in general not dense inside the Banach spaces $L^\infty(X,V)$, for any measure space $X$ and any Banach space $V$. \\
            Density of $S^\infty(X,V)$ inside $L^\infty(X,V)$ is in fact guaranteed only when $V$ is a separable Banach space. \\  
            The following elementary example illustrates the situation. 
            Consider the measure space $\N$ endowed with the discrete measure $\mu_{\N}(n) = 2^{-n}$, for all $n \in \N$.
            Then, $(\N,\mu_{\N})$ has finite total measure and we can regard it as an object of the category $\Meas$.
            Consider the Hilbert space $l^2(\N)$ and define a function $f \colon \N \rightarrow l^2(\N)$ by setting $f(n) = e_n$, for any $n \in \N$, where $\{e_n\}_{n \in \N}$ is the standard ortonormal basis of $l^2(\N)$.
            It is easy to check that $f$ is strongly measurable, Bochner integrable and satisfies $\norm{f}_{L^\infty\left( \N,l^2(\N) \right)} = 1$.
            Suppose $g \in L^\infty\left( \N,l^2(\N) \right)$ is any simple function.
            Since $g$ has finite range, at least one of its level-sets, say $A_x = g^{-1}(x)$ for some $x \in l^2(\N)$, must be infinite inside $\N$.
            For any elements $n_1,n_2 \in A_x$, we then have:
            \begin{equation*}
                \begin{split}
                    \sqrt{2} &= \norm{e_{n_1} - e_{n_2}}_{l^2(\N)} \\
                    &= \norm{f(n_1) - f(n_2)}_{l^2(\N)} \\
                    &= \norm{\left( f(n_1) - x \right) - \left( f(n_2) - x \right)}_{l^2(\N)} \\
                    &= \norm{\left( f(n_1) - g(n_1) \right) - \left( f(n_2) - g(n_2) \right)}_{l^2(\N)} \\
                    &\leq \norm{f(n_1) - g(n_1)}_{l^2(\N)} + \norm{f(n_2) - g(n_2)}_{l^2(\N)}.
                \end{split}
            \end{equation*}
            Taking the supremum of the last inequality over all $n_1,n_2 \in \N$, we finally obtain
            \begin{equation*}
                \frac{\sqrt{2}}{2} \leq \norm{f - g}_{L^\infty\left( \N,l^2(\N) \right)},
            \end{equation*}
            hence simple functions cannot be dense in $L^\infty\left( \N,l^2(\N) \right)$.
        \end{remark}

        \begin{remark}
            Following-up on Remark \ref{rem:connection to Leinster}, we briefly comment on the relation between our approach and the approach of \cite{Leinster2023}.
            Despite the fact that our framework provides, as a particular case, a universal characterisation of the Banach spaces of equivalence classes of $\K$-valued $L^p$ Bochner integrable functions, under equality almost everywhere, for any $p \in [1,\infty)$, there does not appear to be any simple relation between the categories $\Bembp$ or $\Bp$ and the categories $\mathscr{B}_{\mathrm{emb}}^p \subset \Ban^{\embMeas}$ or $\mathscr{B}^p \subset \Ban^{\opMeas}$, introduced in \cite[\S3]{Leinster2023}.  
        \end{remark}

        Theorem \ref{thm:universal properties of Bochner integrable functions} uniquely characterises the Banach spaces of equivalence classes of $L^p$ Bochner integrable functions, under equality almost everywhere, up to isometric isomorphisms, since the morphisms of the category $\Ban$ are contractive maps.

        The universal properties of the objects $(L^p,\sigma) \in \Bembp$, for any $p \in [1,\infty)$, naturally yield a unique characterisation of Bochner integration, as follows. \\
        Set $p = 1$ and consider the bifunctor $\proj \colon \embMeas \times \Ban \rightarrow \Ban$ defined on objects and morphisms by projecting to the second entry. 
        More concretely, for each object $(X,V) \in \embMeas \times \Ban$ and for any morphism $(\iota,T) \colon (X,V) \rightarrow (Y,W)$ in $\embMeas \times \Ban$, we define $\proj(X,V) = V \in \Ban$ and
        \begin{equation}
        \label{eq:action of proj on morphisms}
            \left( \proj \right)_{(\iota,T)} = T \colon V \rightarrow W.
        \end{equation}
        In addition, let $\tau$ assign, to any object $(X,V) \in \embMeas \times \Ban$, the map:
        \begin{equation}
        \label{eq:definition of tau}
            \tau_{(X,V)} = \mu_X(X) \id_V \colon V \rightarrow \proj(X,V) = V.
        \end{equation}
        It is easy to check that the pair $(\proj,\tau)$ defines an object in the category $\Bemb$.
        In fact, for any morphism $(\id_X,T) \colon (X,V) \rightarrow (X,W)$ in $\embMeas \times \Ban$, axiom \ref{ax:compatibility with linear contractive maps} follows directly from formulas \eqref{eq:action of proj on morphisms} and \eqref{eq:definition of tau}:
        \begin{equation*}
            \begin{split}
                \left( \proj \right)_{(\id_X,T)} \circ \tau_{(X,V)} &= T \circ \mu_X(X) \id_V \\
                &= \mu_X(X) \id_W \circ ~ T \\
                &= \tau_{(X,W)} \circ T.
            \end{split}
        \end{equation*}
        For any morphisms $(\alpha,\id_V) \colon (Y,V) \rightarrow (X,V)$ and $(\beta,\id_V) \colon (Z,V) \rightarrow (X,V)$ in the category $\embMeas \times \Ban$, where $Y \xrightarrow{~\alpha~} X \xleftarrow{~\beta~} Z$ are a couple of complementary embeddings, the validity of axiom \ref{ax:compatibility with complementary embeddings in the normed case} is obtained by:  
        \begin{align}
                \tau_{(Y,V)}^{(X,V)} + \tau_{(Z,V)}^{(X,V)} &= \left( \proj \right)_{(\alpha,\id_V)} \circ \tau_{(Y,V)} + \left( \proj \right)_{(\beta,\id_V)} \circ \tau_{(Z,V)} \notag \\
                &= \mu_Y(Y) \id_V + \mu_Z(Z) \id_V \notag \\
                &= \left[ \mu_X\left( \alpha(Y) \right) + \mu_X\left( \beta(Z) \right)\right] \id_V \label{eq:use that embeddings preserve measures} \\
                &= \mu_X(X) \id_V \label{eq:use that alpha and beta are complementary embeddings} \\
                &= \tau_{(X,V)} \notag,
        \end{align}
        where equation \eqref{eq:use that embeddings preserve measures} follows from the fact that embeddings preserve measures and equation \eqref{eq:use that alpha and beta are complementary embeddings} from the fact that $X = \alpha(Y) \cup \beta(Z)$ and $\alpha(Y) \cap \beta(Z) = \emptyset$ (see Definition \ref{def:complementary embeddings}).
        Regarding the validity of axiom \ref{ax:boundedness of c}, formula \eqref{eq:definition of tau} yields directly $\norm{\tau_{(X,V)}}_{\mathcal{L}(V,V)} \leq \mu_X(X)$, for any object $(X,V) \in \embMeas \times \Ban$.
        Finally, for any morphisms $(\alpha,\id_V) \colon (Y,V) \rightarrow (X,V)$ and $(\beta,\id_V) \colon (Z,V) \rightarrow (X,V)$ in the category $\embMeas \times \Ban$, where $Y \xrightarrow{~\alpha~} X \xleftarrow{~\beta~} Z$ are a couple of complementary embeddings, for all $u \in \proj(Y,V) = V$ and $w \in \proj(Z,V) = V$, axiom \ref{ax:norm inequality wrt complementary embeddings} is simply given by the triangle inequality in the Banach space $V$
        \begin{equation*}
            \begin{split}
                \norm{\left( \proj \right)_{(\alpha,\id_V)} u + \left( \proj \right)_{(\beta,\id_V)} w}_V = \norm{u + w}_V \leq \norm{u}_V + \norm{w}_V.
            \end{split}
        \end{equation*}

        We are now ready to prove, as a last step, the third main result of this paper.
        
        \begin{proposition}
        \label{prop:uniqueness of the Bochner integral}
            The unique morphism $(L^1,\sigma) \rightarrow (\proj,\tau)$ in the category $\Bemb$ is given, component-wise, by the Bochner integral:
            \begin{equation*}
                \int_{(X,V)} \colon L^1(X,V) \rightarrow V,
            \end{equation*}
            for any object $(X,V) \in \embMeas \times \Ban$. 
        \end{proposition}

        \begin{proof}
            By statement (i) of Theorem \ref{thm:universal properties of Bochner integrable functions}, it suffices to show that $\int_{-}$ is a morphism in the category $\Bemb$.
            For any object $(X,V) \in \embMeas \times \Ban$, the fact that the component $\int_{(X,V)} \colon L^1(X,V) \rightarrow V$ is a linear contractive map follows directly from linearity of the Bochner integral and from the basic inequality (see, for example, \cite[Theorem II.2.4]{Diestel/Uhl1977}) 
            \begin{equation*}
                \norm{\int_X f d\mu_X}_V \leq \int_X \norm{f}_V d\mu_X, \qquad \forall f \in L^1(X,V).
            \end{equation*}
            In order to check naturality of the transformation $\int_{-} \colon L^1 \Rightarrow \proj$, we need to show that, for any morphism $(\iota,T) \colon (X,V) \rightarrow (Y,W)$ in $\embMeas \times \Ban$, the diagram
            \begin{equation*}
                \begin{tikzcd}
                    L^1(X,V) \ar[rr, "L^1_{(\iota,T)}"] \ar[dd, "\int_{(X,V)}"']  &   & L^1(Y,W) \ar[dd, "\int_{(Y,W)}"] \\
                        &   & \\
                    \proj(X,V) \ar[rr, "\left( \proj \right)_{(\iota,T)}"]   &   & \proj(Y,W)
                \end{tikzcd}
            \end{equation*}
            commutes in the category $\Ban$.
            Combining formula \eqref{eq:action of Sp on emb-morphisms} with formula \eqref{eq:action of proj on morphisms}, for any element $f \in L^1(X,V)$, we obtain: 
            \begin{align}
                \left( \int_{(Y,W)} \circ ~ L^1_{(\iota,T)} \right) (f) &= \int_Y L^1_{(\iota,T)} f d\mu_Y \notag \\
                &= \int_{\iota(X)} T \circ f \circ \iota^{-1} d\mu_{\iota(X)} \label{eq:use definition of L1_iotaT} \\
                &= \int_X T \circ f d\mu_X \label{eq:use that iota is measure-preserving} \\
                &= T \left( \int_X f d\mu_X \right) \label{eq:use properties of the Bochner integral} \\
                &= \left( \left( \proj \right)_{(\iota,T)} \circ \int_{(X,V)} \right) (f), \notag
            \end{align}
            where equation \eqref{eq:use definition of L1_iotaT} follows from the definition of $L^1_{(\iota,T)} f$, equation \eqref{eq:use that iota is measure-preserving} follows from the fact that embeddings preserve measures and equation \eqref{eq:use properties of the Bochner integral} is one of the basic properties of the Bochner integral (cf.~\cite[\S V.5, Corollary 2]{Yosida1995}). \\
            Finally, we need to check the requirement that $\int_{(X,V)} \circ ~ \sigma_{(X,V)} = \tau_{(X,V)}$, for any object $(X,V) \in \embMeas \times \Ban$.
            Using formulas \eqref{eq:definition of sigma} and \eqref{eq:definition of tau}, for any vector $v \in V$, we obtain indeed: 
            \begin{align*}
                \left( \int_{(X,V)} \circ ~ \sigma_{(X,V)} \right) (v) &= \int_X \sigma_{(X,V)}(v) d\mu_X \\
                &= \int_X v \cdot \chi_X d\mu_X \\
                &= \mu_X(X) \cdot v \\
                &= \tau_{(X,V)} (v).
            \end{align*}
            This concludes the proof of the proposition.
        \end{proof}

        \begin{remark}
            In Proposition \ref{prop:uniqueness of the Bochner integral}, we have regarded the pair $(\proj,\tau)$ as an object of the category $\Bemb$.
            But the pair $(\proj,\tau)$ can also be regarded as an object of the category $\B$, where the bifunctor $\proj \colon \opMeas \times \Ban \rightarrow \Ban$ is the obvious extension of the bifunctor $\proj \colon \embMeas \times \Ban \rightarrow \Ban$ and the assignment $\tau$ is left unchanged.
            In particular, the validity of axiom \ref{ax:compatibility with measure-preserving maps in the normed case}, for any morphism $(m^{\mathrm{op}},\id_V) \colon (X,V) \rightarrow (Y,V)$ in the category $\opMeas \times \Ban$, where $m \colon Y \rightarrow X$ is a measure-preserving map, can be checked as follows:
            \begin{align}
                \left( \proj \right)_{(m^{\mathrm{op}},\id_V)} \circ \tau_{(X,V)} &= \id_V \circ \, \mu_X(X) \id_V \notag \\
                &= \mu_X(X) \id_V \notag \\
                &= \mu_Y(Y) \id_V \label{eq:using that m is measure-preserving} \\
                &= \tau_{(Y,V)} \notag,
            \end{align}
            where eqation \eqref{eq:using that m is measure-preserving} follows from the fact that $m$ is a measure-preserving map. \\
            From statement (ii) of Theorem \ref{thm:universal properties of Bochner integrable functions} we then obtain that the unique morphism $\int_{-} \colon (L^1,\sigma) \rightarrow (\proj,\tau)$ in the category $\Bemb$ appearing in Proposition \ref{prop:uniqueness of the Bochner integral} is in fact a morphism in the category $\B$.
            The additional naturality of $\int_{-}$ with respect to morphisms of the type $\left( m^{\mathrm{op}},T \right) \colon (X,V) \rightarrow (Y,W)$ in the category $\opMeas \times \Ban$, where $m \colon Y \rightarrow X$ is a measure-preserving map, corresponds, for any $f \in L^1(X,V)$, to the formula for integration under a change of variables:
            \begin{align}
                \left( \int_{(Y,W)} \circ ~ L^1_{\left( m^{\mathrm{op}},T \right)} \right) (f) &= \int_Y L^1_{\left( m^{\mathrm{op}},T \right)} f d\mu_Y \notag \\
                &= \int_Y T \circ f \circ m d\mu_Y \label{eq:use the action of Sp on measure-preserving maps} \\
                &= \int_X T \circ f d\mu_X \notag \\
                &= T \left( \int_X f d\mu_X \right) \notag \\
                &= \left( \left( \proj \right)_{\left( m^{\mathrm{op}},T \right)} \circ \int_{(X,V)} \right) (f), \notag
            \end{align} 
            where, in equality \eqref{eq:use the action of Sp on measure-preserving maps}, we used formula \eqref{eq:action of Sp on measure-preserving maps}.
        \end{remark}

\section{Conclusions and outlook}
\label{sec:conclusions and outlook}

    The framework we have developed to obtain a universal characterisation of the Banach spaces of equivalence classes of $L^p$ Bochner integrable functions, under equality almost everywhere, represents a broader formulation of the framework developed in \cite{Leinster2023} to characterise the Banach spaces of equivalence classes of $\K$-valued Lebesgue integrable functions, under equality almost everywhere. \\
    However, there are even more general theories of integration (such as, for example, the Pettis integral \cite{Talagrand1984}), which are based on the concept of duality and are applicable not only to the spaces of strongly measurable functions, but also to the spaces of weakly measurable functions. 
    Although these theories lie beyond the scope of this work, we believe that the proposed approach may serve as a valid starting point for characterising these more general theories of integration as well. \\
    As explained in the introduction, our interest in the Bochner integral is motivated primarily by the study of the signature of paths taking values in vector spaces. 
    Therefore, our plan for future research is to apply the results obtained so far to the characterisation of the iterated integrals of paths taking values in vector spaces.

\section*{Acknowledgements}

    We would like to express our gratitude to Sylvie Paycha for first proposing the idea for this project and for her constant support. 
    We would like to thank Tom Leinster and Darrick Lee for providing us with inspiration through their work, for stimulating conversations and for their insightful comments. 
    Thanks also to Jonathan Taylor for kindly sharing his technical expertise in category theory on various occasions over a cup of coffee. \\
    F.~M.~gratefully acknowledges funding by the Deutsche Forschungsgemeinschaft (DFG, German Research Foundation) -- CRC/TRR 388 ``Rough Analysis, Stochastic Dynamics and Related Fields'' -- Project ID 516748464.

\section*{AI disclosure}
    
    Parts of this manuscript were prepared with the assistance of artificial intelligence tools. 
    The authors retain full responsibility for the content, accuracy, and mathematical correctness of the manuscript.

\printbibliography

\end{document}